\documentclass[a4paper]{gtart}

\usepackage[inner=20mm, outer=20mm, textheight=245mm]{geometry}

\usepackage{psfrag, graphicx, subfigure, epsfig,color}

\usepackage{amsmath, amssymb, latexsym, euscript}

\usepackage{mathptmx, wrapfig}

\usepackage[colorlinks,citecolor=blue,linkcolor=blue, backref=page]{hyperref}
\usepackage[nameinlink]{cleveref}

\usepackage[colorinlistoftodos]{todonotes}
\theoremstyle{plain}
\newtheorem{theorem}{Theorem}
\newtheorem*{theorem*}{Theorem}

\newtheorem{proposition}[theorem]{Proposition}
\newtheorem{corollary}[theorem]{Corollary}

\theoremstyle{definition}

\newtheorem*{definition*}{Definition}

\theoremstyle{remark}
\newtheorem{remark}[theorem]{Remark}

\numberwithin{equation}{section}

\DeclareMathOperator{\gap}{gap}

\newcommand{\tri}{\mathcal T}
\newcommand{\farey}{\mathcal F}

\newcommand{\meridian}{\mathfrak{m}}
\newcommand{\longitude}{\mathfrak{l}}

\begin{document}

\title{Complexity of 3--manifolds obtained by Dehn filling}
\author{William Jaco, J. Hyam Rubinstein, Jonathan Spreer and Stephan Tillmann}

\begin{abstract}
Let $M$ be a compact 3--manifold with boundary a single torus.
We present upper and lower complexity bounds for closed 3--manifolds obtained as even Dehn fillings of $M.$
As an application, we characterise some infinite families of even Dehn fillings of $M$ for which our method determines the complexity of its members up to an additive constant. The constant only depends on the size of a chosen triangulation of $M$, and the isotopy class of its boundary. 

We then show that, given a triangulation $\tri$ of $M$ with $2$--triangle torus boundary, there exist infinite families of even Dehn fillings of $M$ for which we can determine the complexity of the filled manifolds with a gap between upper and lower bound of at most $13 |\tri| + 7.$ This result is bootstrapped to obtain the gap as a function of the size of an ideal triangulation of the interior of $M$, or the number of crossings of a knot diagram. We also show how to compute the gap for explicit families of fillings of knot complements in the three-sphere. The practicability of our approach is demonstrated by determining the complexity up to a gap of at most 10 for several infinite families of even fillings of the figure eight knot, the pretzel knot $P(-2,3,7)$, and the trefoil.
\end{abstract}

\primaryclass{57M25, 57N10}
\keywords{3--manifold, minimal triangulation, layered triangulation, complexity, Farey tessellation, slope norm}
\makeshorttitle


\section{Introduction}

We define the \textbf{complexity} of a triangulable manifold $M$ to be the minimum number of top-dimensional simplices in a semi-simplicial triangulation of $M.$ For closed, irreducible manifolds in dimension three - the focus of this work - this notion coincides for all but three manifolds with Matveev's complexity~\cite{Matveev-complexity-1990} that was defined in terms of spines.
The notion of complexity is an important organising principle when studying manifolds through the lens of low-dimensional topology. For any given $n,d \in \mathbb{N}$, there is only a finite number of $d$--manifolds of complexity $\leq n$; and systematic census enumeration using triangulations naturally generates all triangulations up to a certain complexity. In this very precise sense, complexity is to manifolds, what the crossing number is to knots.

Determining the complexity of a given manifold is a hard problem in general. Before we discuss closed $3$-manifolds, note that several results on the complexity of $3$--manifolds with boundary exist. See, for instance, \cite{Frigerio03Complexity,Ishikawa-construction-2016,Jaco-ideal-2020,RST21NewFamily} for complexity bounds on ideal triangulations, and \cite{JJSTBoundsOnNormSurfs} for complexity bounds on triangulations with real boundary. 

In the closed case, early lower bounds on complexity use an analysis of homology and fundamental groups \cite{Matveev-2001-homologygroups,Pervova-2008-algebraiccomplexity}, or hyperbolic volume computations \cite{Matveev-2009-asymptotic,Petronio-2009-volumecomputations}. Bounds in terms of hyperbolic volume are only sharp in very special cases \cite{Fominykh-census-2016,Vesnin-three-2014}. A recent approach developed by Lackenby and Purcell~\cite{Lackenby-2019-complexity} gives complexity bounds for hyperbolic $3$--manifolds that fibre over the circle using the monodromy of the bundle.
Census enumeration trivially determines the complexity of all manifolds in a given census and hence a lower bound for all manifolds that do not appear in that census. Currently, this determines the complexity of all closed, irreducible, orientable $3$--manifolds up to complexity $13$ \cite{Matveev-2020-tabulation} - an impressive algorithmic and computational achievement. 

Upper bounds usually arise from the explicit construction of triangulations, and the difficulty lies in closing the gap between upper and lower bounds. For instance, for the Weber-Seifert dodecahedral space, it is currently only known that its complexity lies between 14 (since it does not appear in the current census) and 23 (by an explicit construction~\cite{Burton-Weber-2012}).

In this paper, we build on observations on least-genus surface representatives of $\mathbb{Z}_2$-homology classes to produce new complexity bounds. This is the only approach currently known to provide exact complexity bounds for infinite families of closed $3$-manifolds -- more precisely, spherical 3--manifolds \cite{Jaco-minimal-2009,Jaco-2013-Norm-Pt1} and $3$--manifolds modelled on $\widetilde{\text{SL}_2(\mathbb{R})}$ \cite{Jaco-norm-2020}. It also certifies complexity for some infinite classes of cusped hyperbolic $3$--manifolds \cite{Jaco-ideal-2020,RST21NewFamily}. 


\medskip

Our new contributions to this line of work are complexity bounds up to a practical additive constant for infinite families of closed $3$--manifolds obtained by Dehn filling.
More precisely, we prove 

\newtheorem*{thm:basic}{Theorem~\ref{thm:basic}}
\begin{thm:basic}
Let $M$ be an orientable, compact, irreducible 3-manifold with boundary an incompressible torus, and let $\tri$ be a triangulation of $M$ with a $2$--triangle torus boundary. Then there exist infinite families of even Dehn-fillings $M(\alpha_k)$ of $M$,  $\alpha_k \in \mathbb{Q} \cup \{\infty\}$, $k \geq 0$, such that  
  \[ 2k \,\, \leq \,\, c(M(\alpha_k)) \,\, \leq \,\, 2k + 13 | \tri|  + 7 . \]
\end{thm:basic}

In particular, for each once-cusped hyperbolic 3--manifold $M$ of finite volume, this gives an infinite family of closed hyperbolic 3--manifolds whose volumes converge to the volume of $M$ and whose complexity is known up to an additive constant that only depends on $M.$ We remark that at the time of writing, there is no infinite family of \emph{closed} hyperbolic 3--manifolds for which the complexity is known exactly.

The \textbf{gap} in the above bound, denoted by $\gap(M(\alpha_k)),$ is the difference between the upper and lower bound on the complexity of $M(\alpha_k).$ Hence the above theorem provides an infinite family where the gap is constant equal to $13 | \tri|  + 7.$ In particular, 
\[
\frac{\gap(M(\alpha_k))}{c(M(\alpha_k))} \in O \left (\frac{1}{c(M(\alpha_k))} \right)
\]

We extend \Cref{thm:basic} to similar statements with input an ideal triangulation (\Cref{thm:ideal}), or a knot-diagram (\Cref{thm:knotbasic}). Neither of these three results explicitly describes the filling slopes $\alpha_k.$ 
Knots in the three-sphere have a canonical framing, and our methods can be used to determine explicit bounds for infinite families of even fillings, where the gap is only a function of the number of crossings of a knot projection. A sample result of this form is

\newtheorem*{thm:constantgap}{Theorem~\ref{thm:constantgap}}
\begin{thm:constantgap}
Let $K$ be a knot distinct from the unknot, and let $D$ be a reduced diagram of $K$ with $n$ crossings. Moreover, let $M = \mathbb{S}^3 \setminus N(K)$ be the knot exterior of $K$ with the standard framing on $\partial M$.

Let $m_0 = 1401(n-1)$, $n_0 = m_0 \cdot 2^{7m_0 + 2}$, and $k>n_0.$ 
  Then we have for the complexity of $M(2k/1)$
  \[ 2(k - n_0) \,\, \leq \,\, c(M(2k/1)) \,\, \leq \,\, m_0 + 2k - 1. \]
\end{thm:constantgap}
The proof of \Cref{thm:constantgap} can be adapted to give a bound for other families of even fillings and those families giving rise to a bound up to an additive constant are easily identified. 
 Since every $3$--manifold can be obtained from Dehn filling on a link in the 3-sphere \cite{Lickorish-1962-LWTheorem1,Wallace-1963-LWTheorem2}, \Cref{thm:constantgap} can be applied in a quite broad setting.
 
The reader should think of the theoretical results discussed so far as a flexible toolkit that can be applied to specific families of examples.
While \Cref{thm:constantgap} cites a very large constant, this constant is much smaller in practical settings. We present three extended examples, analysing various families of Dehn fillings of the figure eight knot in \Cref{ssec:fig8}, the  pretzel knot $P(-2,3,7)$ in \Cref{ssec:pretzel}, and the trefoil in \Cref{ssec:trefoil}. In several cases of infinite families of fillings allowing a constant gap, this gap is in the single digits.
The goal of this extended list of examples is to demonstrate that, given a knot and very little extra information, we can determine practical upper and lower complexity bounds for infinite families of even Dehn fillings using out-of-the-box software such as {\tt Regina} \cite{regina} or {\tt SnapPy} \cite{SnapPy}.

\textbf{Acknowledgements }
Jaco is partially supported by the Grayce B. Kerr Foundation.
Research of Rubinstein, Spreer and Tillmann is supported in part under the Australian Research Council's Discovery funding scheme (project number DP190102259). 
The main result of this paper was conceived whilst the authors were supported through the programme “Research in Pairs”
by the Mathematisches Forschungsinstitut Oberwolfach in 2017. The authors would like to thank
the staff at MFO for an excellent collaboration environment.


\section{Background}
\label{sec:background}

We refer to \cite{Jaco-norm-2020} for background and standard definitions used in this paper, and only recall the following two key definitions. 
Given a closed $3$--manifold $M$, we define the \textbf{complexity} of $M$ to be the minimum number $c(M)$ of tetrahedra in a triangulation of $M.$ The \textbf{norm} $|| \; \phi \; ||$ of a non-trivial class $\phi \in H^1 (M, \mathbb{Z}_2 )$ is the negative of the maximal Euler characteristic of a properly embedded surface $S$, no component of which is a sphere or projective plane, representing the Poincar{\'e} dual of $\phi.$ 
%

\subsection{$3$--manifolds with torus boundary and the Farey tessellation}
\label{fig:bdry}

Let $M$ be an orientable, compact, irreducible 3-manifold with $\partial M$  consisting of a single incompressible torus boundary component. Let $(\meridian,\longitude)$ be a framing of $\partial M.$ Since $\partial M$ is incompressible and has abelian fundamental group, we have $\text{im}(\pi_1(\partial M)\to \pi_1(M))\cong \pi_1(\partial M) \cong H_1(\partial M, \mathbb{Z}) .$ As is usual for the torus, we freely move between isotopy, homotopy, homology classes depending on context and most efficient notation.
Hence, for an isotopy class of non-trivial simple closed loops on the boundary torus $\alpha \in \text{im}(\pi_1(\partial M)\to \pi_1(M))$, we refer to the non-trivial primitive class $\alpha \in H_1 (\partial,\mathbb{Z})$, $\alpha = \meridian^q \longitude^p$, as a \textbf{slope}, and vice versa. A slope is an \textbf{even slope}, if it maps to zero in $H_1 (M, \mathbb{Z}_2).$

\begin{proposition}[\protect{Corollary 10 in \cite{JRST21SlopeNorm}}]
Let $\alpha \in \text{im}(\pi_1(\partial M)\to \pi_1(M))$ be a slope.
There is a properly embedded surface $S$ in $M$ with $[\partial S] = \alpha$ if and only if $\alpha$ is an even slope.
\end{proposition}

This motivates the definition of the \textbf{norm} of an even slope $\alpha$ in $M$ as
\[ || \; \alpha \; || = \min \{\ - \chi(S) \ \mid \ \text{$S$ is a properly embedded surface in $M$ with }  [\partial S] = \alpha \} .\]
We say that $S$ is \textbf{taut} for $\alpha$ if $S$ is connected, $[\partial S] = \alpha$ and $|| \; \alpha \; || = - \chi(S).$ 

Let $\tri$ be a $0$-efficient triangulation of $M.$ Then $\tri$ has a single vertex, and the induced triangulation $\tri_{\partial}$ of $\partial M$ has exactly two triangles and necessarily contains this vertex.
We briefly sketch how the fundamental normal surfaces $\{F_i\}$ of $\tri$, together with the dual graph of the Farey tessellation -- as an organising principle of boundary slopes on $\tri_{\partial}$ -- can be used to compute the slope norm for an arbitrary even slope $\alpha$ of $M.$ We refer to \cite[Section 2]{JRST21SlopeNorm} for details.

\begin{figure}[h]
    \centerline{\includegraphics[width=0.6\textwidth]{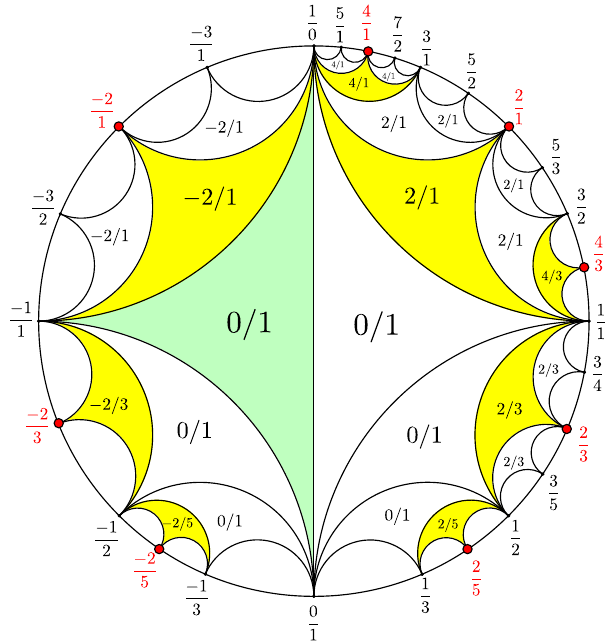}}
    \caption{The Farey tessellation. \label{fig:Farey}}
  \end{figure} 

Consider the Farey tessellation $\farey$ associated with the framing $(\meridian,\longitude)$ for $\partial M$, see \Cref{fig:Farey}. Each ideal triangle $\tau$ corresponds to an isotopy class of 1-vertex triangulations of $\tri_{\partial}.$ Its ideal vertices are labelled with the slopes $(\alpha, \beta, \gamma)$ of the edges for $\tri_{\partial}$, and each ideal triangle is labelled with its unique even slope, say $\alpha$, which is referred to as the \textbf{even slope} of $\tau.$ Marked are the base triangle in green, and the \textbf{canonical} triangles for the even slopes in yellow. A canonical triangle is characterised by the property that the ideal vertex carrying the even slope lies between the two other ideal vertices on the boundary of the tessellation.


The dual graph to the Farey tessellation $\Gamma (\farey)$ is an infinite trivalent tree. Travelling accross an arc in $\Gamma (\farey)$ corresponds to flipping an edge in $\tri_{\partial}$ yielding another isotopy class of $2$-vertex triangulations of the torus. On the level of the triangulation $\tri$ of $M$, this edge flip is realised by \textbf{layering} an extra tetrahedron on top of $\tri_{\partial}$, increasing the size of the triangulation by one. 

Every isotopy class of $2$--triangle triangulations of $\partial M$ can be realised as the boundary of some triangulation of $M$, and hence every even slope of $\partial M$ is an edge in some triangulation of $M.$

Related to this organising principle, there are two measures of distance on $\Gamma (\farey)$ of interest to us. Let $\tau$ and $\tau'$ be two ideal triangles of the Farey tessellations. By abuse of notation, we refer to their corresponding nodes in $\Gamma (\farey)$ by $\tau$ and $\tau'$ as well. Let $\alpha$ and $\alpha'$ be the even slope labels of $\tau$ and $\tau'$ respectively. By $d_{\farey}(\tau,\tau')$ we denote the length of the unique shortest path in $\Gamma (\farey)$ between $\tau$ and $\tau'.$ By $d(\alpha,\alpha')$ we denote the number of distinct even slope labels we see on the unique shortest path in $\Gamma (\farey)$ from a triangle labeled $\alpha$ to a triangle labeled $\alpha'$ minus one. Moreover, for $\alpha$ an arbitrary even slope we define $d([0],\alpha) = \infty.$ By construction, we have $2 d(\alpha,\alpha') \leq d_{\farey}(\tau,\tau'),$ and this bound is best possible.

In \cite{JRST21SlopeNorm} it is shown that for some even slope $\alpha$, the slope norm of $\alpha$ equals
\begin{equation}
\label{eq:defnorm}
|| \; \alpha \; || = -\chi(S) = \min_{F_i} \{\; - \chi(F_i) + d([\partial F_i], \alpha) \;\}
\end{equation}
where the minimum is taken over all fundamental surfaces $F_i$ of $\tri.$ Note that it is enough to minimise over the set of incompressible and $\partial$-incompressible fundamental surfaces of $M$ with connected essential boundary.

Let $M(\alpha)$ be the Dehn filling of $M$ along $\alpha.$ Moreover, let $S\subset M$ be taut for $\alpha.$ Consider the union of $S$ and the meridian disk of the filling torus in $M(\alpha)$, and denote its Poincar{\'e} dual by $\phi_{\alpha} \in H^1 (M(\alpha),\mathbb{Z}_2).$ By construction we have $|| \; \phi_\alpha \; || = || \; \alpha \; || -1.$

\section{Complexity bounds on even Dehn fillings}
\label{sec:bounds}

In this section we first deduce lower and upper bounds for the complexity of $M(\alpha).$ We then describe infinite families of Dehn fillings for which the gap between these bounds is constant. 

\subsection{Lower Bound} 
\label{ssec:lowerbound}

A \textbf{balanced lens space} is a lens space $M$ with even fundamental group that satisfies $c(M) = 1 + 2 || \ \varphi \ ||,$ where $\varphi$ is a generator for $H^1(M;\mathbb{Z}_2).$
With the setup from \Cref{sec:background} and the following theorem from \cite{Jaco-norm-2020}, we directly obtain a lower bound for the complexity of $M(\alpha).$

\begin{theorem}[Corollary 2 in \cite{Jaco-norm-2020}]
\label{thm:closed}
Let $M$ be a closed orientable, irreducible, connected 3--manifold not homeomorphic with a balanced lens space and suppose that $0 \neq \varphi \in H^1(M;\mathbb{Z}_2).$ Then $c(M) \ge 2 + 2 || \ \varphi \ ||.$
\end{theorem}

\begin{corollary}
\label{eq:lowerbound}
Let $M$ be an orientable, compact, irreducible 3-manifold with boundary an incompressible torus, and let $\alpha$ be an even filling slope of $M$, such that $M(\alpha)$ is not a balanced lens space. Then 
\begin{equation}
c(M(\alpha)) \geq 2 || \; \alpha \; ||,
\end{equation}
where $|| \; \alpha \; ||$ denotes the slope norm of $\alpha $ in $M.$
\end{corollary}

\begin{proof}
Since $M(\alpha)$ is not a balanced lens space, it follows from \Cref{thm:closed} that $c(M(\alpha)) \geq 2 + 2 || \; \phi_\alpha \; ||  = 2 + 2(|| \; \alpha \; || -1 ) = 2|| \; \alpha \; ||.$
\end{proof}


\subsection{Upper bound} 
\label{ssec:upperbound}

Let $M$ be an orientable, compact, irreducible 3-manifold with boundary an incompressible torus.
Fix a framing $(\meridian,\longitude)$ on $\partial M$ and let $\tri$ be a triangulation of $M$ with a one-vertex two-triangle torus boundary $\tri_\partial.$ Let $\tau$ be the node in $\Gamma (\farey)$ corresponding to the isotopy class of $\tri_\partial.$

We can turn $\tri$ into a triangulation of a Dehn-filling of $M$ by \textbf{folding} $\tri_\partial$ over one of its three boundary edges. That is, the two triangles in $\tri_\partial$ are identified in such a way that one obtains a M\"obius band. The edge that one folds over becomes the boundary of the M\"obius band, and the other two edges are identified. See \Cref{fig:torus}. The kernel of the induced map on fundamental group from the torus to the M\"obius band is generated by the associated filling slope. This can be worked out from the identification of the two edges of $\tri_\partial$ by the folding operation as follows.

\begin{figure}
    \centerline{\includegraphics[width=0.6\textwidth]{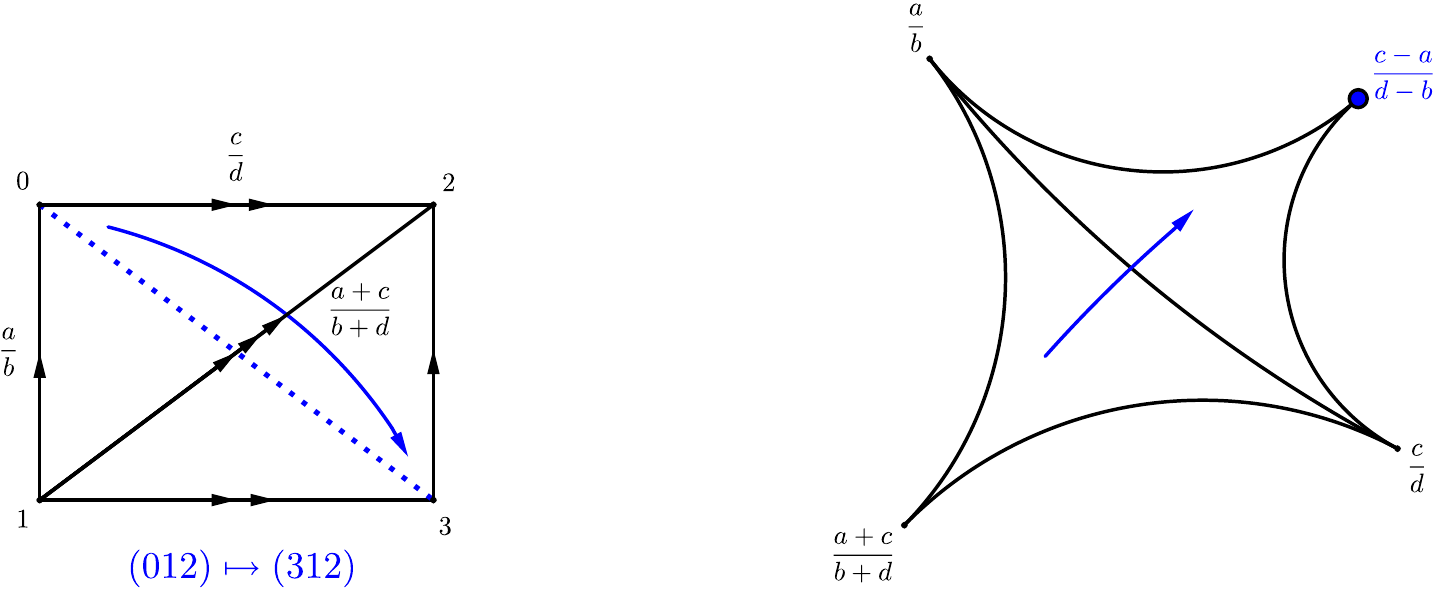}}
    \caption{Left: Torus boundary $\tri_\partial$ of isotopy class $(a/b,c/d,(a+c)/(b+d)).$ The arrow indicates the folding over the diagonal, the dotted line indicates the target filling slope. Right: Corresponding ideal triangle(s) in the Farey tessellation. The arrow indicates source and target triangle, the bold vertex indicates the target filling slope. \label{fig:torus}}
\end{figure}

Suppose we fold over the diagonal edge in \Cref{fig:torus} on the left. This yields the filling slope $(c-a)/(d-b)$, which is the opposite diagonal, and hence a triangulation of the manifold $M((c-a)/(d-b)).$ 
Folding over the even edge produces a closed non-orientable surface of the same Euler characteristic as the negative of the current slope norm. This means there are two ways to relate the slope norm of an even boundary slope $\alpha$ to the $\mathbb{Z}_2$-norm of the associated class in the Dehn filled manifold $M(\alpha)$: 

\begin{enumerate}
  \item Layering on an ideal triangle labelled $\alpha$, thereby adding an additional saddle (decreasing Euler characteristic by one) - and then capping off the bounded taut surface with a disk in $M(\alpha)$ (increasing the Euler characteristic by one).
  \item Layering on one ideal triangle before a triangle labelled $\alpha$, and closing the bounded taut surface by antipodal identification (leaving the Euler characteristic invariant).
\end{enumerate}

Given $\tri$ and a target even Dehn filling slope $\alpha$, we can use the Farey tessellation to work out how to layer on $\tri_\partial$, to obtain a triangulation of $M(\alpha)$ via folding: From $\tau$, the node of $\Gamma (\farey)$ corresponding to the isotopy class of $\tri_\partial$, layer on $\tri_\partial$ following the unique shortest path from $\tau$ to one step before a node labelled $\alpha$ (if $\tau$ is already labelled $\alpha$, perform one layering to obtain an isotopy class of the boundary not labelled $\alpha$). Denote this target node by $\tau'.$ Now folding over the even boundary edge yields a triangulation $\tri_\alpha$ of $M(\alpha).$ See \Cref{fig:torus} on the left for $\alpha = (c-a)/(d-b).$ 

By construction, we have 
\begin{equation}
\label{eq:upperbound}
c(M(\alpha)) \leq |\tri_\alpha| = |\tri| + d_{\farey} (\tau,\tau')
\end{equation} 

Note that this upper bound does not only depend on $|\tri|$, but also on the isotopy class of $\tri_\partial$ (in \Cref{eq:upperbound}, this information is incorporated in $\tau$). This plays a role in the bound derived in \Cref{sec:generalknotexterior}, and, again, in \Cref{sec:example} where we look at different triangulations of the figure eight knot complement and the Pretzel knot $P(-2,3,7)$ to minimise the gap between upper and lower bounds for Dehn fillings of this manifold. 

\begin{remark}
  Note that, whenever we want to calculate the norm of an even boundary slope, we must work with a $0$-efficient triangulation. The reason behind this is that, this way, for every boundary slope bounding an incompressible and $\partial$-incompressible surface, a norm-minising surface with this slope is amongst the fundamental surfaces in the triangulation, see \cite{JRST21SlopeNorm}.
   However, here and in the following sections we only need a guarantee that for every boundary slope of an incompressible and $\partial$-incompressible surface, there exists a fundamental normal surface in the triangulation with a single boundary component realising this slope. By virtue of \cite[Proposition 3.7 and its corollaries]{Jaco-decision-2003}, this is satisfied as soon as the triangulation has a $2$--triangle torus boundary.
\end{remark}

\subsection{Families of filling slopes with constant gap} 
\label{ssec:constant}

Let $\tri$ be a triangulation of $M$ with $2$--triangle torus boundary. Let $\{ F_i\}$ be the fundamental normal surfaces of $\tri$, and let $\mathcal{S}$ be the finite subset of vertices of $\Gamma (\farey)$ associated with the boundary slopes of those $\{ F_i\}$ with a single non-trivial boundary curve in $\tri_\partial.$ Denote the vertex of $\Gamma (\farey)$ corresponding to the isotopy class of $\tri_\partial$ by $\tau_0.$ 
Choose a framing $(\meridian,\longitude)$ on $\partial M$ such that $\tau_0 = \tau (0/1,1/0,-1/1).$

In $\Gamma (\farey)$, starting at node $\tau_0 = \tau (0/1,1/0,-1/1)$, follow any infinite path $\tau_k$, $k \geq 0$, in $\Gamma (\farey)$ where the even slope labels change at every second node. Equivalently, follow a path that alternates between nodes corresponding to white and yellow triangles, see \Cref{fig:Farey}.   

Truncate this path such that it starts with the last node $\tau' \in \mathcal{S}$ (with even slope label $\alpha'$), and refer to every even slope $\alpha$ as \textbf{admissible} if \textbf{(a)} $\alpha$ is an even slope label on the path and \textbf{(b)} the previous even slope label $\alpha''$ of a node $\tau''$ on the path is still on the truncated portion of the path. Note that we have 
\begin{equation}
\label{eq:line0}
2 d(\alpha',\alpha'') \leq d_{\farey}(\tau',\tau'') \leq 2 d(\alpha',\alpha'') + 1.
\end{equation}

It follows that we have for the difference between upper and lower bounds for $M(\alpha)$ not a balanced lens space, $\alpha$ admissible:

\begin{align}
0\,\,\leq \,\,|\tri_\alpha| - 2 || \; \alpha \; || &= |\tri| + d_{\farey} (\tau,\tau'') - 2 || \; \alpha \; || \label{eq:line1}\\
&\leq |\tri| + d_{\farey} (\tau,\tau') + d_{\farey} (\tau',\tau'') - 2 || \; \alpha \; || \label{eq:line2}\\
&\leq |\tri| + d_{\farey} (\tau,\tau') + d_{\farey} (\tau',\tau'') - 2 d(\alpha',\alpha'') \label{eq:line3}\\
&\leq |\tri| + d_{\farey} (\tau,\tau') + 1 \label{eq:line4}
\end{align}

Here, \Cref{eq:line1} is the difference between \Cref{eq:upperbound} and $2 || \; \alpha \; ||$. This is non-negative by virtue of \Cref{eq:lowerbound}. \Cref{eq:line2} is a simple application of the triangle inequality for $d_{\farey}.$ \Cref{eq:line3} follows from the setup of the path between $\tau'$ and $\tau'' $, the definition of $|| \; \cdot \; ||$ in \Cref{eq:defnorm}, and the assumption that the slope norm of the slope corresponding to the second even slope label on the truncated path is $0.$ Finally, \Cref{eq:line4} implements the more pessimistic case of \Cref{eq:line0}.

Since neither $|\tri|$ nor $d_{\farey} (\tau,\tau')$ depend on the choice of admissible slope $\alpha$, this determines the complexity of the infinite family of closed manifolds $\{ M(\alpha) \}$, $\alpha$ admissible, up to a constant.

Note that, if all members of $\{ M(\alpha) \}$ are hyperbolic, we can decrease the constant by one, accounting for the fact that the norm of the first slope must be positive. Also note, that this bound can be improved by looking at different triangulations $\tri$ with different isotopy classes of $\tri_\partial.$ In particular, the choice of triangulation affects both $|\tri|$ and $d_{\farey} (\tau,\tau').$


\section{An upper bound for the constant gap}
\label{sec:generalknotexterior}

As above, let $M$ be an orientable, compact, irreducible 3-manifold with boundary an incompressible torus. Moreover, as above, let $\tri$ be a triangulation of $M$ with a $2$--triangle torus as boundary. In this section we compute upper bounds for $|\tri|$ and $d_{\farey} (\tau, \tau')$ from \Cref{ssec:constant}, and hence the gap in complexity, for infinite families of Dehn fillings with constant gap of $M.$ 
Our bounds only depend on $|\tri|$ (\Cref{thm:basic}), the number of tetrahedra in an ideal triangulation $\tri'$ of the interior of $M$ (\Cref{thm:ideal}), or the number of crossings of a knot diagram $D$ of a knot $K \subset \mathbb{S}^3$ in the case $M = \mathbb{S}^3 \setminus N(K)$  (\Cref{thm:knotbasic}).

In addition, we give an improvement of \Cref{thm:knotbasic}, where we have control over the knot theoretic framing of $\partial M.$ This allows us to determine constant gaps for explicitly chosen families of even Dehn fillings of knot exteriors only depending on the crossing number of a diagram of a knot (\Cref{thm:constantgap}). 

\begin{theorem}
  \label{thm:basic}
  Let $M$ be an orientable, compact, irreducible 3-manifold with boundary an incompressible torus, and let $\tri$ be a triangulation of $M$ with a $2$--triangle torus boundary. Then there exist infinite families of even Dehn-fillings $M(\alpha_k)$ of $M$,  $\alpha_k \in \mathbb{Q} \cup \{\infty\}$, $k \geq 0$, such that  
  \[ 2k \,\, \leq \,\, c(M(\alpha_k)) \,\, \leq \,\, 2k + 13 | \tri|  + 7 . \]
\end{theorem}

\begin{proof}
Since $M$ is a $3$--manifold with a single torus boundary component, every incompressible and $\partial$-incompressible surface in $M$ has one of finitely many boundary slopes \cite{Hatcher-boundary-1982}. Since the triangulation $\tri$ has exactly two boundary triangles, for every boundary slope of an incompressible and $\partial$-incompressible surface, there exists a fundamental normal surface in $\tri$ with a single boundary component realising this slope \cite[Proposition 3.7 and its corollaries]{Jaco-decision-2003}. 

Let $|\tri| = n.$ By the work of Hass, Lagarias, and Pippenger \cite{HLP}, a fundamental surface can have at most $n\cdot 2^{7 n + 2}$ normal arcs per boundary normal arc type. Choose a framing on $M$ with one edge of $\tri_\partial$ following the meridian $\meridian$ and one following the longitude $\longitude$ such that the isotopy class of $\tri_\partial$ is $(0/1,1/0,-1/1).$ It follows that $\partial F$ intersects each of $\meridian$ and $\longitude$ at most $2\, n\cdot 2^{7 n + 2}$ times.

\medskip

Construct an infinite path in the dual of the Fary tessellation $\Gamma (\farey)$: Starting at node $\tau(0/1,1/0,-1/1)$ go to a node $\tau'$ with associated even slope $\alpha = 2p/q$ with $2p > 2\, n\cdot 2^{7 n + 2} = n \cdot 2^{7n+3}.$ Then proceed away from $\tau(0/1,1/0,-1/1)$ and $\tau'$ with a new even slope in every second node. In the language of \Cref{ssec:constant}, we call the truncated path starting at $\tau'$ the admissible path: $\tau'$ is the last node on the path possibly still contained in $\mathcal{S} \subset \Gamma (\farey)$. Denote the associated even slopes of the admissible path by $\alpha_k$, $k \geq 0$, $\alpha_0 = \alpha.$

\paragraph*{Claim:} $c (M (\alpha_k)) \geq 2k$ for the complexity of $M(\alpha_k).$

\paragraph*{Proof of the claim:} Let $\tau$ be a node in $\mathcal{S} \subset \Gamma (\farey)$, and let $-\chi$ be the smallest negative Euler characteristic of a surface with slope the even slope of $\tau$. Following \cite{JRST21SlopeNorm}, we compute the slope norm of $\alpha_k$ by taking the minimum of $- \chi$ plus the number of even slopes ($\neq \alpha_k$) observed on a path in $\Gamma (\farey)$ from $\tau$ to a node with even slope label $\alpha_k$, ranging over all nodes $\tau \in \mathcal{S}.$ (Note that it would be enough to only consider nodes $\tau$ associated to the slope of an incompressible, $\partial$-incompressible surface in $M$.) By construction, this path must pass through $\tau'$. Otherwise, we have a fundamental normal surface of $\tri$ with slope $2p/q$, $2p > n \cdot 2^{7n+3}$, intersecting the edge of $\tri$ running along $\longitude$ more than $n \cdot 2^{7n+3}$ times, a contradiction. This implies that we see at least $k+1$ distinct even slopes on the admissible path, and we have $||\alpha_k|| \geq k.$ It then follows from \Cref{eq:lowerbound} and the fact that our triangulation of $M (\alpha_k)$ cannot be a balanced lens space that we have $c (M (\alpha_k)) \geq 2k.$

Note that, because of the non-empty sequence of layerings along the Fibonacci path described below, our triangulation of $M (\alpha_k)$ is guaranteed not to be a balanced lens space. 

\medskip

On the other hand, we can triangulate $M(\alpha_k)$ by starting with $\tri$ and layering tetrahedra along the shortest path of $\tau(0/1,1/0,-1/1)$ to $\tau'.$ We then need $2 k-1$ more tetrahedra to layer onto the boundary to reach a boundary isotopy class that yields a triangulation $\tri_{\alpha_k}$ of $M(\alpha_k)$ by folding the even boundary edge.
Hence, in order to compute a bound for the gap up to which we can determine the complexity of $M(\alpha_k)$, it remains to bound $d(\tau(0/1,1/0,-1/1),\tau')$ (cf. \Cref{eq:upperbound}). 

The shortest path from $\tau(0/1,1/0,-1/1)$ to $\tau'$ (a node with even slope coefficients larger than $n \cdot 2^{7n+3}$) is the following path:

\begin{align*}
  & \tau (0/1,1/0,-1/1) \\
  & \tau (1/1,1/0,0/1) \\
  & \tau (2/1,1/1,1/0) \\
  & \tau (3/2,2/1,1/1) \\
  & \tau (5/3,3/2,2/1) \\
  & \ldots \\
  & \tau (F_{\ell}/F_{\ell-1},F_{\ell-1}/F_{\ell-2},F_{\ell-2}/F_{\ell-3}) \\
  & \tau (F_{\ell+1}/F_{\ell},F_{\ell}/F_{\ell-1},F_{\ell-1}/F_{\ell-2}) \\
\end{align*}

Here $F_0=0 $, $F_1=1$, $F_i = F_{i-1}+F_{i-2}$, $i\geq 2$, is the Fibonacci sequence. As described above, we choose $\tau'$ and associated even slope $\alpha = F_{\ell+1} / F_{\ell}$ where $F_{\ell+1}$ is an even Fibonacci number such that $F_{\ell+1} >  n \cdot 2^{7n+3}.$ By construction, the length of the path from $\tau(0/1,1/0,-1/1)$ to this $\tau'$ is exactly $\ell.$

We have $F_i = \lfloor \frac{\phi^i}{\sqrt{5}} +\frac12 \rfloor$ for $\phi = \frac{1+\sqrt{5}}{2} \approx 1.618$. Observe that $\frac{\phi^\ell }{2} \geq \lfloor \frac{\phi^i}{\sqrt{5}} +\frac12 \rfloor$ for $\ell \geq 2$. Since $n\geq 1$, we have $n \cdot 2^{7n+3} \geq 1024$ and $\ell \geq 2$ can safely be assumed. It follows that we need to bound $\ell$ such that $ \frac{\phi^\ell}{2}  >  n \cdot 2^{7n+3}.$ This translates to

\[ \ell > \frac{1}{\log_2 (\phi)} \cdot \left ( \log_2 \left ( n \right) +1+ 7n+3 \right ).\]

Since $n > \log_2 \left (n \right)$ we can instead compute $\ell$ to satisfy $\ell >  \frac{8n+4}{\log_2 (\phi)} \approx 1.4404201 (8n+ 4).$ Since every third Fibonacci number is even, $\ell = 12n + 8$ satisfies the bound.

Altogether, this means we can triangulate $M(\alpha_k)$ by starting with $\tri'$, layering $12n + 8$ tetrahedra on its boundary to obtain a triangulation with boundary isotopy class $(F_{\ell+1}/F_{\ell},F_{\ell}/F_{\ell-1},F_{\ell-1}/F_{\ell-2})$, followed by layering $2k-1$ additional tetrahedra on its boundary before folding over the boundary.

We thus have the upper bound
\[ c(M(\alpha)) \leq n + 12n + 8 + 2k-1 = 13n + 7 + 2k \]
This completes the proof.
\end{proof}

\begin{corollary}
  \label{thm:ideal}
  Let $M$ be an orientable, compact, irreducible 3-manifold with boundary an incompressible torus, and let $\tri'$ be an ideal triangulation of the interior of $M.$ Then there exist infinite families of even Dehn-fillings $M(\alpha_k)$ of $M$,  $\alpha_k \in \mathbb{Q} \cup \{\infty\}$, $k \geq 0$, such that  
  \[ 2k \,\, \leq \,\, c(M(\alpha_k)) \,\, \leq \,\, 2k + \frac13 \left (143 |\tri'| + 151 \right ). \]
\end{corollary}

\begin{proof}
  Let $n = |\tri'|$ be the number of ideal tetrahedra in $\tri'.$
  According to \cite[Section 4.4]{jaco03-0-inflations}, inflating the ideal vertex of $\tri'$ along frame $\Lambda$ in the vertex link of $\tri'$ produces a triangulation $\tri$ of the compact core of $M$ with $|\tri| = |\tri'| + e(\Lambda) + \mathbb{X}(\Lambda) + 2.$
  Here, $e(\Lambda)$ is the number of edges in frame $\Lambda$ and $\mathbb{X}$ is a correction term accounting for the fact that conflicting diagonals of quadrilateral faces may be introduced in the inflation process - requiring extra tetrahedra to be inserted.
  
  The vertex link of $\tri'$ is a triangulated torus with $2n$ vertices, $6n$ edges and $4n$ triangles. The frame $\Lambda$ is a graph in the $1$-skeleton of the vertex link with Euler characteristic $-1.$ Hence, $\Lambda$ can have at most $2n+1$ edges and hence $e(\Lambda) \leq 2n+1.$
  
  Recall that edges in $\Lambda$ are normal arcs in triangles $t \subset \tri'.$ Hence, $t$ can contain between zero and three edges of the framing. In the case of two and three edges, inflating at $t$ corresponds to adding a triangulated pyramid over a quadrilateral or a triangulated prism over a triangle. The diagonal in the pyramid can be freely chosen, but for the prism, only six of the eight combinations of diagonals are possible. As a result, for every such $t$ containing three edges of the frame, we may need an additional tetrahedron to flip a conflicting diagonal. In the worst case this adds another $\mathbb{X}(\Lambda) \leq \lfloor \frac{e(\Lambda)}{3} \rfloor \leq \lfloor \frac{2n+1}{3} \rfloor $ tetrahedra to $\tri.$
  
  Altogether we have 

  \[|\tri| \leq n + 2n+1 + \lfloor \frac{2n+1}{3} \rfloor + 2 \leq  \lfloor \frac{11n+10}{3} \rfloor.\]  
  
  Applying \Cref{thm:basic} to $\tri$ proves the result.
\end{proof}

\begin{corollary}
  \label{thm:knotbasic}
  Let $K$ be a knot distinct from the unknot, and let $D$ be a diagram of $K$ with $n$ crossings. Moreover, let $M = \mathbb{S}^3 \setminus N(K)$ be the knot exterior of $K.$
  
  Then there exist infinite families of even Dehn-fillings $M(\alpha_k)$ of $M$,  $\alpha_k \in \mathbb{Q} \cup \{\infty\}$, $k \geq 0$, such that  
  \[ 2k \,\, \leq \,\, c(M(\alpha_k)) \,\, \leq \,\, 2k + \frac13 \left (572 n + 723 \right ). \]
\end{corollary}

\begin{proof}
  A well-known construction due to Weeks \cite[Section 3]{handbookknottheory} produces an ideal triangulation $\tri'$ from an $n$--crossing diagram of a link with one cusp per link component and $4n+4$ tetrahedra.
   Applying the inflation in the proof of \Cref{thm:ideal} to $\tri'$, hence produces a triangulation $\tri$ with $2$--triangle torus boundary with $|\tri| \leq \lfloor \frac{44n+54}{3} \rfloor$ tetrahedra. Applying \Cref{thm:basic} to $\tri$ proves the result.  
\end{proof}

For the final statement of this section, we say that a diagram $D$ of a knot $K$ is \textbf{reduced}, if it does not allow reducing Reidemeister moves of type I or II. We call the pair of essential curves $(\meridian_K, \longitude_K)$ on $\partial M$ the \textbf{knot theoretic framing}, if $\meridian_K$ bounds a disk in $N(K)$, and $\longitude_K$ intersects $\meridian_K$ once and has linking number zero with $K$ in $\mathbb{S}^3.$
Determining the knot theoretic framing first, we can give bounds for explicitly chosen infinite families of Dehn fillings of $M.$ Here, we prove this in the special case of filling slopes $2k/1$ for $k$ sufficiently large.

\begin{theorem}
  \label{thm:constantgap}
  Let $K$ be a knot distinct from the unknot, and let $D$ be a reduced diagram of $K$ with $n$ crossings. Moreover, let $M = \mathbb{S}^3 \setminus N(K)$ be the knot exterior of $K$, and let $m_0 = 1401(n-1)$, $n_0 = m_0 \cdot 2^{7m_0 + 2}$, and $k>n_0.$
  
  Then we have for the complexity of $M(2k/1)$
  \[ 2(k - n_0) \,\, \leq \,\, c(M(2k/1)) \,\, \leq \,\, m_0 + 2k - 1. \]
\end{theorem}

\begin{remark}
  Note that the focus on fillings $2k/1$ is arbitrary. Using the identical method, we can compute explicit bounds for other families of filling slopes with constant gap (as presented in \Cref{ssec:constant}).
\end{remark}

\begin{proof}
The proof of this statement has the following main steps and ingredients:

\begin{enumerate}
  \item Construct a triangulation $\tri$ of $M$ with boundary $\tri_\partial$ a torus containing $\meridian_K$ and $\longitude_K$ as simple closed loops of edges meeting in a single vertex.
  \item Turn $\tri$ into a triangulation $\tri'$ with boundary $\tri'_\partial$ a two-triangle torus of isotopy class $(0/1,1/0,-1/1)$ with respect to the knot theoretic framing. In particular, one boundary edge runs along the meridian, one boundary edge runs along the longitude of the knot theoretic framing of $\partial M.$ This step takes up the bulk of the proof.
  \item As in \Cref{thm:basic}, invoke Hatcher \cite{Hatcher-boundary-1982}, Jaco and Sedgwick \cite{Jaco-decision-2003}, and Hass Lagarias and Pippenger \cite{HLP}.
  \item Use the Farey tessellation and the known isotopy class of $\tri'_\partial$ to show $|| 2k/1 || \geq k - c$ for some constant $c$. The complexity of $M(2k/1)$ is bounded above by the size of $\tri'$ and the length of a path in the dual graph of the Farey tessellation.
\end{enumerate}

\paragraph*{The triangulation $\tri$:} We apply a slightly revised construction of \cite[Lemma 7.1 and Lemma 7.2]{HLP} to $D.$ In \cite{HLP}, the authors first turn $D$ into a maximal planar graph (with crossings as vertices), possibly by introducing extra vertices at bigons of $D$ - which they call \textbf{special vertices} - and edges. Since, in our case, $D$ is reduced, the number of special vertices is bounded above by $n$ itself, and we have for the total number of vertices in the subdivided planar graph $m \leq 2n$ (instead of $m \leq 5n$ in \cite{HLP}). 
This yields a maximal planar graph, or planar triangulation, with $\leq 4n-5$ bounded triangular regions -- or triangles. We take the union of these triangles cross an interval to obtain a collection of $\leq 4n-5$ triangular prisms, denoted by $P.$

Combining \cite[Lemma 7.1 and Lemma 7.2]{HLP} we only consider one layer of such prisms $P$ (instead of three in \cite{HLP}) and subdivide them into $14$ tetrahedra each (with one vertex in the centre of each quadrilateral, coning over a vertex in the centre of $P$) to obtain a triangulation $P'$ of $P$ with at most $14(4n-5) = 56n-70$ tetrahedra, and at most $2(4n-5)+12 = 8n +2$ triangles in its boundary $\partial P'.$ See \Cref{fig:Pprime} for details about constructing $P'.$ Coning $\partial P'$ to a single point at infinity this yields a triangulation $S$ of the $3$-sphere with $\leq 64n -68 < 64(n-1)$ tetrahedra. 

\begin{figure}
\includegraphics[height=2.3cm]{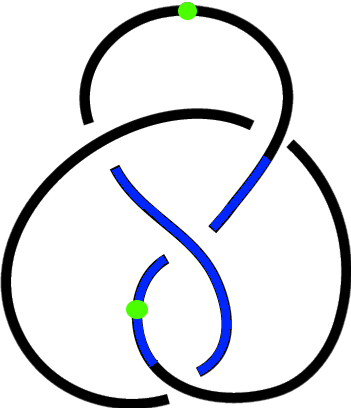} \quad \quad
\raisebox{-.6cm}{\includegraphics[height=3.0cm]{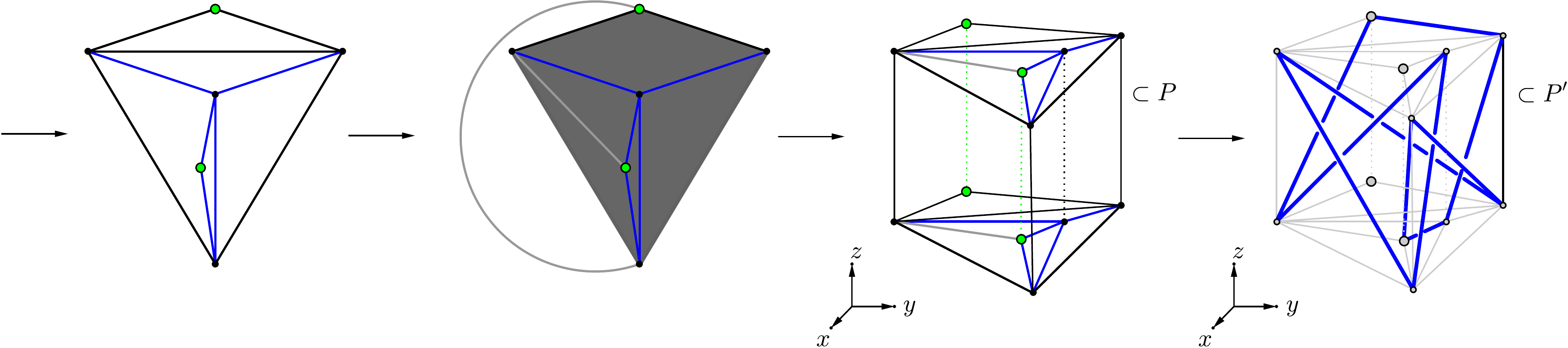}}
\caption{From $D$ to $P'.$ In the rightmost picture of $P'$, subdivisions are mostly omitted for readability. The blue lines denote the edges representing $K.$ \label{fig:Pprime}}
\end{figure}

By construction, $S$ contains the knot $K$ as a simple closed loop $L$ in its $1$-skeleton: Follow the top (bottom) edge of a prism for an arc of $D$ from an overcrossing (undercrossing) to and overcrossing (undercrossing). Follow the two edges in a diagonal of a quadrilateral prism face for an arc in $D$ from an overcrossing to an undercrossing, or an undercrossing to an overcrossing respectively. Whenever we encounter a special vertex, we first follow the appropriate edge of a trangular prism face before following the appropriate diagonal of the next quadrilateral prism face. It follows that the length of $L$ is bounded above by $6n.$

Placing $P'$ into $\mathbb{R}^3$ with the planar triangulations parallel to the $xy$-plane, and the interval in $z$-direction (see \Cref{fig:Pprime} on the right), we can see that $D$ can be recovered from $L$ by projecting a regular neighbourhood of $P'$ in $S$ into the $xy$-plane from $z$-direction.

Removing a small regular neighbourhood of $L$ from $S$ produces either tetrahedra with neighbourhoods of zero, one, two, or three vertices removed, or tetrahedra with the neighbourhood of one edge, and zero or one vertices removed. To see this note that, \textbf{(a)} since $D$ is reduced and hence does not admit any reducing R2 moves, at most one edge per tetrahedron in $S$ is in $L$, and \textbf{(b)} each tetrahedron in $S$ has exactly one vertex that either lies at the centre of a triangular prism, or at infinity, and hence away from $L.$

Triangulating the boundary of these truncated tetrahedra produces at most $16$ triangles (see \Cref{fig:Tri} bottom row for the case realising $16$ triangles, all other types of truncated tetrahedra can be triangulated with fewer tetrahedra, see, for instance, \Cref{fig:Tri} top row), and coning these over a single vertex in its centre produces a triangulation $\tri$ of the knot exterior of $K$ with at most $16 \cdot 64(n-1) = 1024 (n-1)$ tetrahedra. Note that at most three triangles per triangulated boundary of a truncated tetrahedron are in the boundary $\tri_\partial$ of $\tri.$ See \Cref{fig:Tri} for some details about constructing $\tri.$ 

\begin{figure}
\includegraphics[width=\textwidth]{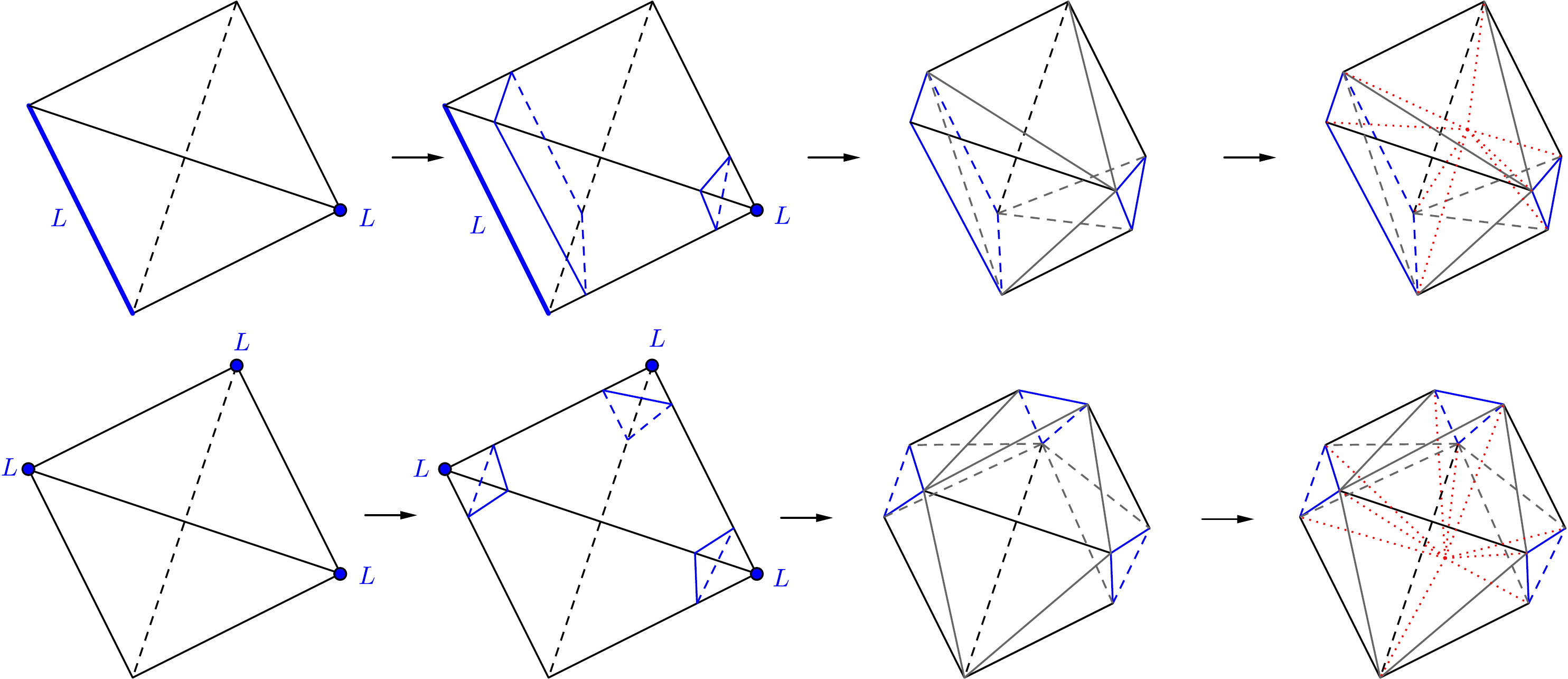}
\caption{Removing a small neighbourhood of $L$ from a tetrahedron followed by triangulating the resulting truncated tetrahedron. Top row: $L$ meets the tetrahedron in an edge and a vertex. This results in a subdivision into $9+3=12$ tetrahedra. Bottom row: $L$ meets the tetrahedron in three vertices. This results in $13+3=16$ tetrahedra. \label{fig:Tri}}
\end{figure}

Looking at the construction of $\tri$ and its boundary, we can identify the geometric meridian $\meridian_K$ of the knot exterior as a loop of six edges in the link of a special vertex. If no special vertex exist, we can create one at the beginning of the construction, and since we assume that we have as many special vertices as original vertices in our construction, this does not change our bound. We can also identify the geometric longitude $\longitude_K$ as a simple closed path in $\tri_\partial$: we simply run along edges in direction of $L$, and realise linking number $0$ with $L$ by walking around meridian curves at non-special vertices as needed. Since $\meridian_K$ lives in a neighbourhood of a special vertex, $\meridian_K$ and $\longitude_K$ are edge-disjoint and meet in a single vertex. 

\paragraph*{The triangulation $\tri'$:} In the next step of the construction, we turn $\tri$ into a triangulation $\tri'$ with a one-vertex, two-triangle boundary torus of isotopy class $(0/1,1/0,-1/1).$ That is, with one of its three boundary edges running parallel to $\meridian_K$, and another one running parallel to $\longitude_K.$

From our calculations about the number of triangles in the truncated tetrahedra (see above) we conclude that $\tri_\partial$ has at most $3\cdot 64(n-1) = 192(n-1)$ triangles and hence at most $9/2 \cdot 64(n-1) = 288 (n-1)$ edges and, since it is a torus, $96(n-10)$ vertices. In particular, its average vertex degree is $6$ and we can always find a vertex $v$ with degree $\leq 6$ in $\tri_\partial.$ 

If $v$ is of degree $1$ or $2$, we can layer two tetrahedra or one tetrahedron respectively onto the triangles adjacent to $v$, as shown in \Cref{fig:lowdegrees}, to turn $v$ into a vertex of degree $3.$ Note that this is always possible since $\tri_\partial$ is not a sphere (and hence the boundary of the two triangles around a vertex of degree $2$ must consist of two distinct edges). 
If $\meridian_K$ or $\longitude_K$ ran through $v$ (only possible in the case that $v$ initially was of degree $2$), we can find a shorter curve on the boundary of the altered triangulation isotopic to the original one.

\begin{figure}
\includegraphics[width=\textwidth]{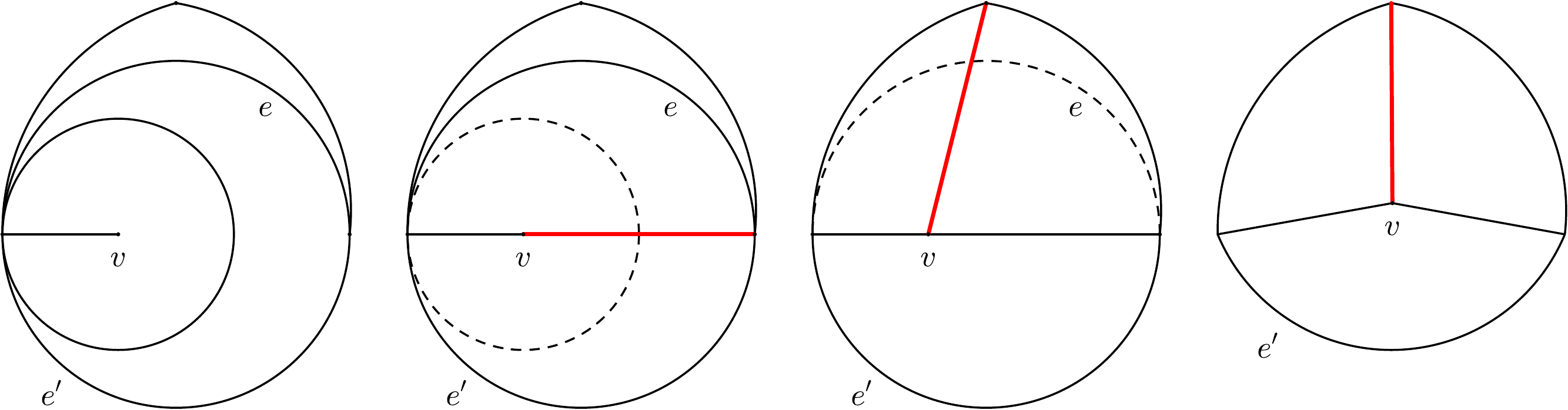}
\caption{Turning a degree one vertex (left), or a degree two vertex (second from left) into a vertex of degree $3$ with two or one flips respectively. Note that edges $e$ and $e'$ must be distinct because $\tri_\partial$ is not a $2$-sphere. \label{fig:lowdegrees}}
\end{figure}

If $v$ is of degree $4$, $5$, or $6$, we have three main cases.

\paragraph*{Case 1} The simple closed paths following $\meridian_K$ and $\longitude_K$ do not run through $v.$ In this case we have three subcases:

\begin{description}
  \item[Case 1.1] All triangles of $\tri_\partial$ contain $v$ at most once. In this case we can add one tetrahedron, two, or three tetrahedra respectively onto the triangles adjacent to $v$, as shown in \Cref{fig:higherdegrees} on the left, to turn $v$ into a vertex of degree $3.$
  \item[Case 1.2] There exists a triangle containing $v$ twice, but no triangle contains $v$ three times. In this case at least two triangles contain $v$ twice and, locally, we must have the picture shown in  \Cref{fig:higherdegrees} on the righthand side. Moreover, since $\tri_\partial$ is not a sphere, $v$ must be of degree at least $5.$ Gluing one tetrahedron, as shown in \Cref{fig:higherdegrees} on the right, decreases the degree by $2$, and causes $v$ to have two fewer triangles containing $v$ twice (actually, since the degree of $v$ is at most $6$, then all remaining triangles must be distinct). 
  \item[Case 1.3] There exists a triangle occuring three times. Then either we have a $1$-vertex $2$--triangle torus and the simple closed paths following $\meridian_K$ and $\longitude_K$ pass through $v$, or the degree of vertex $v$ must be a least $7.$ Either way a contradiction.   
\end{description}

\paragraph*{Case 2} One of $\meridian_K$ or $\longitude_K$ runs through $v.$ Without loss of generality, let $\meridian_K$ contain $v$, and let $\longitude_K$ be disjoint of $v.$ Since $v$ is disjoint of $\longitude_K$, it follows that $\meridian_K$ is of length at least two, and intersects the triangles adjacent to $v$ in exactly two edges, and at least one vertex distinct from $v.$

Since $v$ has degree at most $6$, it occurs in triangles on one side of $\meridian_K$ at most five times. Fix one side. It follows from Case 1.3, that no triangle on this side contains $v$ three times. Moreover, if an edge $e$ contains $v$ twice, it cannot be contained in $\meridian_K.$ Hence, we can layer over $e$ as in Case 1.2 to reduce the degree of $v$ by $2$ without covering an edge contained in $\meridian_K.$ If no triangle contains $v$ more than once, we proceed as in Case 1.1, noting that we can always avoid covering an edge contained in $\meridian_K$ in the process.

\begin{figure}
\centerline{\includegraphics[width=0.7\textwidth]{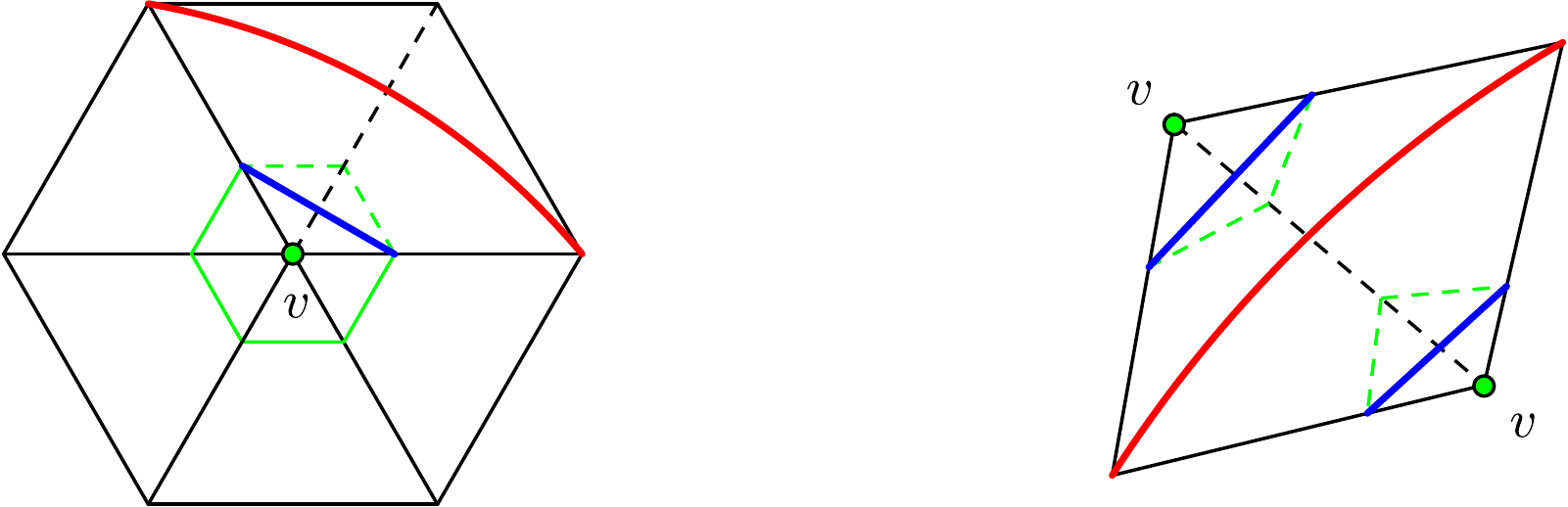}}
\caption{Left: reducing the degree of a boundary vertex $v$ with only distinct triangles around it by one. Right: reducing the degree of a boundary vertex $v$ contained twice in two triangles by two. The layered edge is drawn dashed, the new boundary edge is drawn in red. Vertex linking normal curves are drawn in green. New arcs of the vertex linking curve are drawn in blue, old arcs are dashed. \label{fig:higherdegrees}}
\end{figure}

\paragraph{Case 3} Both of $\meridian_K$ or $\longitude_K$ run through $v.$ In this case, we do not touch this vertex. It must be of degree at least $6$ and hence we can find another vertex of degree at most $6$ to perform the above process on.

\medskip

Altogether, after adding at most three tetrahedra to $\tri$ we obtain a triangulation containing a vertex with exactly three distinct triangles around it. Hence, we can glue one additional tetrahedron to the three triangles surrounding this vertex to produce a triangulation with this vertex no longer in its boundary. Note that this is possible whenever $\tri_\partial$ has more than one vertex, and that the boundary of this new triangulation is smaller by one vertex, three edges, and two triangles. Moreover, by construction, it still contains simple closed paths of edges running along $\meridian_K$ and $\longitude_K$ of equal or shorter length.

Iterating this procedure hence necessarily produces a triangulation $\tri'$ with only two triangles in its boundary. Since $\tri$ has at most $96(n-1)$ vertices in its boundary, the above procedure adds at most $4 \cdot (96(n-1)-1)=384(n-1) -4$ extra tetrahedra to $\tri$ to produce $\tri'.$ It follows that $\tri'$ contains at most $(384+1024) (n-1) -1 < 1408 (n-1) =: m_0$ tetrahedra.

\paragraph{Invoking Hatcher, Jaco-Sedgwick, and Hass-Lagarias-Pippenger:} This part of the proof is completely analogous to the proof of \Cref{thm:basic}. We sketch the argument again for the reader's convenience.

\begin{enumerate}
  \item Due to Hatcher, every incompressible and $\partial$-incompressible surface in $M$ has one of finitely many boundary slopes \cite{Hatcher-boundary-1982};
  \item due to Jaco-Sedgwick, for every boundary slope of an incompressible and $\partial$-incompressible surface, there exists a fundamental normal surface $F$ in $\tri'$ with a single boundary component realising this slope \cite[Proposition 3.7 and its corollaries]{Jaco-decision-2003}; and 
  \item due to Hass, Lagarias, and Pippenger \cite{HLP} $F$ can have at most $n_0 = m_0\cdot 2^{7 m_0 + 2}$ normal arcs per boundary normal arc type. Since the isotopy type of the boundary of $\tri'$ is $(0/1,1/0,-1/1)$, it follows that $\partial F$ intersects each of $\meridian_K$ and $\longitude_K$ at most $2 n_0$ times.
\end{enumerate}

\paragraph*{Deduce upper and lower bounds for the complexity of $M(2k/1)$:} Starting at node $\tau(0/1,1/0,-1/1)$, consider the following path in the dual of the Fary tessellation $\Gamma (\farey)$:
\begin{align*}
  & \tau (0/1,-1/1,1/0) \\
  & \tau (0/1,1/1,1/0) \\
  & \tau (2/1,1/1,1/0) \\
  & \tau (2/1,3/1,1/0) \\
  & \tau (4/1,3/1,1/0) \\
  & \ldots \\
  & \tau ((2k-2)/1,(2k-3)/1,1/0) \\
  & \tau ((2k-2)/1,(2k-1)/1,1/0). 
\end{align*}

Layering ontop of $\tri_\partial'$ along this path, and then folding over the even boundary edge produces a triangulation of $M(2k/1)$ with $m_0 + 2k-1$ tetrahedra.

Let $\tau' = \tau ((2n_0-2)/1,(2n_0-1)/1,1/0)$ be the node corresponding to the isotopy class of the triangulation obtained from $\tri'$ by layering $2n_0 -1$ times on $\tri_\partial'$ along the path above. Moreover, in the language of \Cref{ssec:constant}, call the truncated path starting from $\tau'$ the admissible path. By construction, the even slopes of the admissible path are $2k/1$, $k>  n_0.$

\paragraph*{Claim:} Let $k >  n_0$, then $c (M (2k/1)) \geq 2(k - n_0).$

\paragraph*{Proof of the claim:} Let $\tau$ be a node in $\Gamma (\farey)$ associated to the slope of an incompressible, $\partial$-incompressible surface in $M$, and let $-\chi $ be the negative Euler characteristic of this surface. Following \cite{JRST21SlopeNorm}, we compute the slope norm of $2k/1$ by taking the minimum of $-\chi$ plus the number of even slopes $2k/1$, $k >  n_0$, observed on a path in $\Gamma (\farey)$ from $\tau$ to node $\tau ((2k-2)/1,(2k-1)/1,1/0)$, ranging over all such nodes $\tau.$ By construction, this path must always pass through $\tau'$ observing at least $k-n_0$ distinct even slopes (note that $2k/1$ is not one of them). Hence we have $||\alpha_k|| \geq k-n_0.$ It then follows from \Cref{eq:lowerbound} and the fact that $M (2k/1)$ is not a balanced lens space that we have $c (M (2k/1)) \geq 2(k-n_0).$

On the other hand, we can triangulate $M(2k/1)$ by starting with $\tri'$ and layering $2k-1$ tetrahedra along the shortest path of $\tau(0/1,1/0,-1/1)$ to $\tau ((2k-2)/1,(2k-1)/1,1/0).$ Folding the boundary then produces a triangulation $\tri_{2k/1}$ of $M(2k/1).$

Altogether we have 
\[2(k - n_0) \leq c (M (2k/1)) \leq m_0 + 2k - 1\]
for any $k > n_0.$ This completes the proof.
\end{proof}

It is important to note that, while the constant in \Cref{thm:constantgap} is prohibitively large, it can be made quite small in explicit examples. This is mainly due to the following two observations: \textbf{(a)} boundary edges running parallel to $\meridian_K$ and $\longitude_K$ seem to be common in small triangulations $\tri'$ of the knot exterior, and hence $|\tri'|$ is typically very far from the upper bound given in the proof of \Cref{thm:constantgap}, and \textbf{(b)} fundamental normal surfaces often have boundary patterns with far fewer normal arcs than the bound given by Hass-Lagarias-Pippenger.

We make this precise in \Cref{sec:example}, by providing examples of the actual gap in the cases of the figure eight knot, the $(2,3,7)$--pretzel, and the trefoil. 

\section{Examples}
\label{sec:example}

\subsection{Dehn fillings of the figure eight knot complement}
\label{ssec:fig8}

Throughout this section, let $M$ be the complement of the figure eight knot endowed with the knot-theoretic framing. It is well-known (see, for instance, \cite{Thurston-notes}) that, with respect to this framing, $M$ contains three incompressible, $\partial$-incompressible surfaces: a once-punctured torus with boundary slope $(0,1)$, and two Klein bottles with boundary slopes $(\pm 4,1)$. Let $\alpha$ be an even boundary slope on $\partial M.$ We are interested in the associated Dehn-filling $M(\alpha).$ Note that since the figure eight knot is amphichiral, we have $M(\alpha) \cong M(-\alpha).$

Using a search through the Pachner graph of ideal triangulations of $M$, truncating and simplifying in every step, we obtain $82$ combinatorially inequivalent triangulations of the compact core of $M$, each with ten tetrahedra and a single vertex contained in their two-triangle boundaries. Each of them is $0$--efficient. Let $\tri$ be one of these triangulations, and let $S$ be one of its normal surfaces. The \textbf{boundary pattern} $(a,b,c)$ of $S$ records the intersection numbers of $S$ with the three boundary edges of $\tri_\partial.$ 

Following \cite{JRST21SlopeNorm}, we know that the boundary slopes of both Klein bottles and the punctured torus must appear in the fundamental normal surfaces of $\tri.$ This allows us to determine the isotopy class of the boundary $\tri_\partial.$ As a result, the $82$ triangulations exhibit four distinct isotopy classes in their boundaries. See \Cref{tab:triangulations} for details.

\begin{figure}[htb]
  \begin{tabular}{l|ccccc}
    name & $\#$ triangulations $\tri$ & isotopy class of $\tri_\partial$ & $(0,1)$ $\partial$-pat. & $(4,1)$ $\partial$-pat. & $(-4,1)$ $\partial$-pat. \\ 
    \hline
    Class I & $24$ & $(1/0,1/1,2/1)$ & $(1,1,2)$ & $(1,3,2)$ & $(1,5,6)$ \\
    Class II & $41$ & $(1/0,0/1,1/1)$ & $(1,0,1)$ & $(1,4,3)$ & $(1,4,5)$ \\
    Class III & $ 4$ & $(1/0,3/1,4/1)$ & $(1,3,4)$ & $(1,1,0)$ & $(1,7,8)$ \\
    Class IV & $13$ & $(1/0,2/1,3/1)$ & $(1,2,3)$ & $(1,2,1)$ & $(1,6,7)$ \\
  \end{tabular}
  \caption{The $82$ triangulations of the compact core of the figure eight knot complement with $10$ tetrahedra. The boundary patterns and isotopy class triples follow the same order. \label{tab:triangulations}}
\end{figure}


Fix one of the $82$ triangulations and denote it by $\tri.$ Let $\tau$ be the ideal triangle (node) in the (dual of the) Farey tessellation corresponding to the isotopy class of $\tri_\partial.$ Folding over the even boundary edge of $\tri_\partial$ realises the even Dehn filling with slope the even slope of the ideal triangle $\tau'$ adjacent to $\tau$ opposite the even slope vertex of $\tau$ (see \Cref{ssec:upperbound} and \Cref{fig:torus} for details). Also, recall that layering over boundary edge $e$ of $\tri_\partial$ produces a triangulation with boundary of isotopy class the one corresponding to the adjacent ideal triangle of the Farey tessellation opposite the ideal vertex labelled with the slope of $e$.

In our example, all incompressible and $\partial$-incompressible surfaces and their boundary slopes are known, and we will never encounter triangulations of balanced lens spaces. Hence, following the instructions for folding above, obtaining the lower bound for complexity for $M(\pm \alpha)$ is straightforward: it is twice the smallest number of even slopes encountered on the unique shortest path in the dual of the Farey tessellation from one of the slopes $0/1$ and $\pm 4/1$ to a node labelled $\pm \alpha$ (note that the slope norm of all of $0/1$ and $\pm 4/1$ is one).

At the same time, a triangulation of $M(\pm \alpha)$ obtained from $\tri$ via layering and folding yields the upper bound: It is the size of $\tri$ plus the length of the unique shortest path between $\tau$ and the node before the first node labelled $\alpha.$ Note that this upper bound depends on the choice of triangulation $\tri .$

\begin{figure}
\centerline{\includegraphics[width=0.8\textwidth]{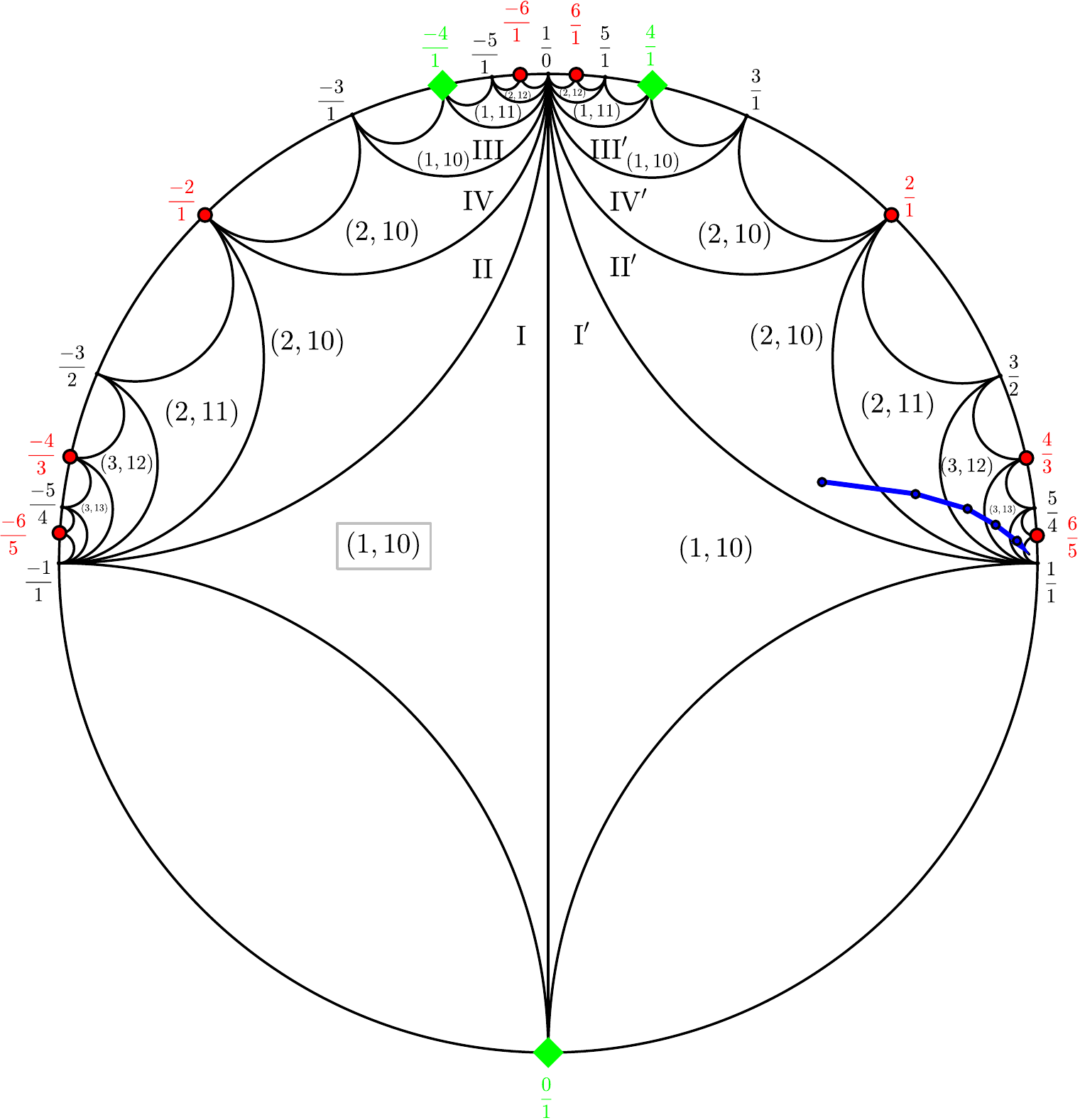}}
\caption{Slope norms and upper bounds per boundary isotopy class of triangulations for the compact core of the figure eight knot complement. For $10$-tetrahedra triangulations, the class numbers I, II, III, and IV are given. Due to the amphichirality of the figure eight knot, reversing the orientation of the meridian (swapping left and right of the picture) gives identical slope norms and upper bounds. Boundary slopes of incompressible and $\partial$-incompressible surfaces are marked in green. \label{fig:fig8example}}
\end{figure}

\Cref{fig:fig8example} shows the first few triangles of the Farey tessellation, locating the four isotopy classes of the boundaries of the $82$ triangulations. In the following, we use this figure to conveniently obtain lower and upper bounds for infinite families of Dehn fillings of $M$:
Every ideal triangle in \Cref{fig:fig8example} is decorated with a pair of numbers $(a,b).$ The first one, $a$, denotes the slope norm of the even slope at this ideal triangle, and the second one, $b$, denotes the minimum size of a (known) triangulation with this isotopy type in the boundary. The latter is obtained from layering, starting from the closest triangulation along the unique shortest path in the dual of the Farey tessellation.

As an example, we can fold over the even edge of a triangulation of Class I (see \Cref{fig:fig8example}) to obtain a ten-tetrahedron triangulation of $M(0/1) = \mathbb{T} \times I / {\tiny \begin{pmatrix} 2 & 1 \\ 1&1 \end{pmatrix}}.$ We refer to the base triangle of Class I as the \textbf{source triangle}, while we refer to the triangle containing slope $0/1$ as the \textbf{target triangle}. The lower bound in complexity for $M(0/1)$ from \Cref{eq:lowerbound} is two (it is twice the first parameter in the target triangle), the upper bound is ten (the second parameter of the target triangle minus one, or, in this case equivalently, the second parameter in the source triangle), while its actual complexity is seven. 

Similarly, we can fold over the even boundary edge of Class IV to obtain a ten-tetrahedron triangulation of $M(4/1) = \operatorname{SFS} [D:(2,1),(2,1)] \cup_{\tiny \begin{pmatrix} 0 & 1\\1 & 0 \end{pmatrix}} \operatorname{SFS} [D:(2,1),(3,1)].$ Again, the lower bound in complexity for $M(4/1)$ from \Cref{eq:lowerbound} is two, the upper bound is ten, while its actual complexity is seven. 

We now consider Class III with boundary isotopy class $(1/0,3/1,4/1)$ and layer on its boundary along the path

\begin{align*}
  & \tau (4/1,3/1,1/0,) \\
  & \tau (4/1,5/1,1/0,) \\
  & \tau (6/1,5/1,1/0,) \\
  & \tau (6/1,7/1,1/0,) \\
  & \ldots \\
  & \tau ((2k-2)/1,(2k-3)/1,1/0) \\
  & \tau ((2k-2)/1,(2k-1)/1,1/0). 
\end{align*}

Folding over the even boundary edge of slope $(2k-2)/1$ of the resulting triangulation yields $M(2k/1)$. This results in a lower bound in complexity from \Cref{eq:lowerbound} of $2k-2$ (note that $M(2k/1)$, $k\geq 3$, is hyperbolic and hence not a lens space), and an upper bound from the triangulation of $2k+5$, $k \geq 3.$ Experimentally, the actual complexity seems to be $2k+4$ (proven for $k=3$). 

Similarly, we can do this for all infinite paths in the dual of the Farey tessellation that have a new even slope in every second step. Some of these infinite paths are straightforward: For an ideal triangle with even slope label $\alpha$, pick the odd slope $\beta$ on the outside (eg., $\alpha = 2/1$ and $\beta = 1/1$ in \Cref{fig:fig8example}). Walk along the infinite path of ideal triangles containing $\beta$ (the blue path in \Cref{fig:fig8example}). This way all even slopes of type $\alpha_k = \alpha \oplus 2k \beta$ are encountered, where $\oplus$ denotes Farey addition (in our example, $\alpha_k = (2k+2)/(2k+1)$). For any infinite family of slopes obtained this way the gap between upper and lower bounds in complexity for $M(\alpha_k)$ can be directly computed from the labels of the starting ideal triangle $\tau$:

Let $(a,b)$ the label of $\tau$ as given in \Cref{fig:fig8example}. Provided that $M(\alpha_k)$ is not a balanced lens space, the lower and upper bounds in complexity for $M(\alpha_k)$, $k >0$, are then given as 
\[ 2(k+a) \quad\leq\quad c(M(\alpha_k)) \quad\leq\quad 2k+b-1.\]

(see above for details and note that these bounds are not valid for the starting point $k=0$ itself). The gap in complexity for $M(\alpha_k)$, $k >0$, is hence $(2k+b-1) - 2(k+a) = b - 2a - 1.$

\begin{remark}
In this example we only consider infinite families of Dehn fillings of $M$ with a constant gap in complexity. More broadly, we can use the same description and method to produce upper and lower bounds in complexity for families of arbitrary even fillings. The only difference is that the gap is potentially widening. This is due to the lower bound only taking into account new even slope labels on the path of fillings, while the upper bound grows with every step.
\end{remark}



%
%
%
\subsection{Even Dehn fillings of the Pretzel $P(-2,3,7)$}
\label{ssec:pretzel}

In this example we compute lower and upper complexity bounds for even Dehn fillings of the Pretzel knot with parameters $-2$, $3$, and $7$. Here we only work with information that is known for large collections of knots. In particular, everything we do in this example can be done for many knots in the {\it SnapPy} \cite{SnapPy} census.

For most of our calculations, software can be used out-of-the box \cite{regina,SnapPy}. Some calculations require small scripts or moderate levels of human interaction. For instance, determining the knot theoretic framing is done using data on exceptional fillings for census knots \cite{Dunfield20Exceptional}, as well as Regina's capabilities to recognise Seifert fibred spaces \cite{regina}.

Let $M$ be the complement of the Pretzel knot $P(-2,3,7)$ (eg. the underlying space of triangulation \texttt{m016} in the {\it SnapPy} \cite{SnapPy} census). We show how to establish the following bounds of gaps between 6 and 8:
\begin{align}	
  2k+2 \leq c(M(-2(k-1)/(2k-1))) \leq 2k + 8, & \quad k \geq 1 \label{eq:pretzel_bounds2}\\
  2k \leq  c(M(-2k/1)) \leq 2k+7,& \quad k \geq 1 \label{eq:pretzel_bounds1}\\ 
  2k+2 \leq c(M(-(6k+2)/(2k+1))) \leq 2k + 10, &\quad  k \geq 1 \label{eq:pretzel_bounds3}
\end{align}

Using the same search through the Pachner graph of ideal triangulations of $M$ as in the previous example, we obtain $93$ triangulations of the compact core of $M$, each $0$--efficient, with ten tetrahedra and a single vertex contained in their two-triangle boundaries. 

Looking at the boundary patterns of the Seifert surface, these $93$ triangulations split into three classes, as indicated in \Cref{tab:pretzel}:
Since each triangulation $\tri$ is $0$-efficient, it follows from \cite[Theorem 5]{JRST21SlopeNorm} that the boundary pattern of the Seifert surface is determined as the boundary pattern of a fundamental orientable normal surface of $\tri$ with boundary a single essential curve (and such a surface always exists). Even more, there must be such a surface with maximum Euler characteristic (realising the genus of the knot). 
  

\begin{figure}[htb]
\begin{center}

\begin{tabular}{l|ccccc}

name & $\#$ triangulations $\tri$ & $\partial$-pattern of Seifert surface \\ 

\hline
Class 1 & $29$ & $(1,17,18)$ \\
Class 2 & $63$ & $(1,19,18)$ \\
Class 3 & $ 1$ & $(1,19,20)$ \\
\end{tabular}
\end{center}
\caption{The $93$ triangulations of the Pretzel knot exterior. \label{tab:pretzel}}
\end{figure}

Let $\tri$ be the unique triangulation of the Pretzel knot exterior with boundary pattern of the Seifert surface $(1,19,20).$ This has \textbf{Regina} \cite{regina} isomorphism signature \texttt{kLvKwIPQcfeghijijjllmgwneflp}. 

We first need to determine the knot theoretic framing. Observe that folding over the even boundary edge of $\tri$ yields the lens space $L(18,5) = M(0/1).$ Moreover, layering over the even boundary edge and then folding back over the resulting degree one even boundary edge yields a graph manifold homeomorphic with $M(-2/1).$ Hence, the even boundary edge of $\tri$ has slope $-2/1$, and the ideal triangle of the Farey tessellation encoding the isotopy class of $\tri_\partial$ is adjacent to a triangle with even slope label $0/1$. It follows that the boundary edges of $\tri$ have slopes $1/0,-2/1,-1/1$, where the order corresponds to the pattern $(1,19,20).$

\begin{figure}
\centerline{\includegraphics[width=.8\textwidth]{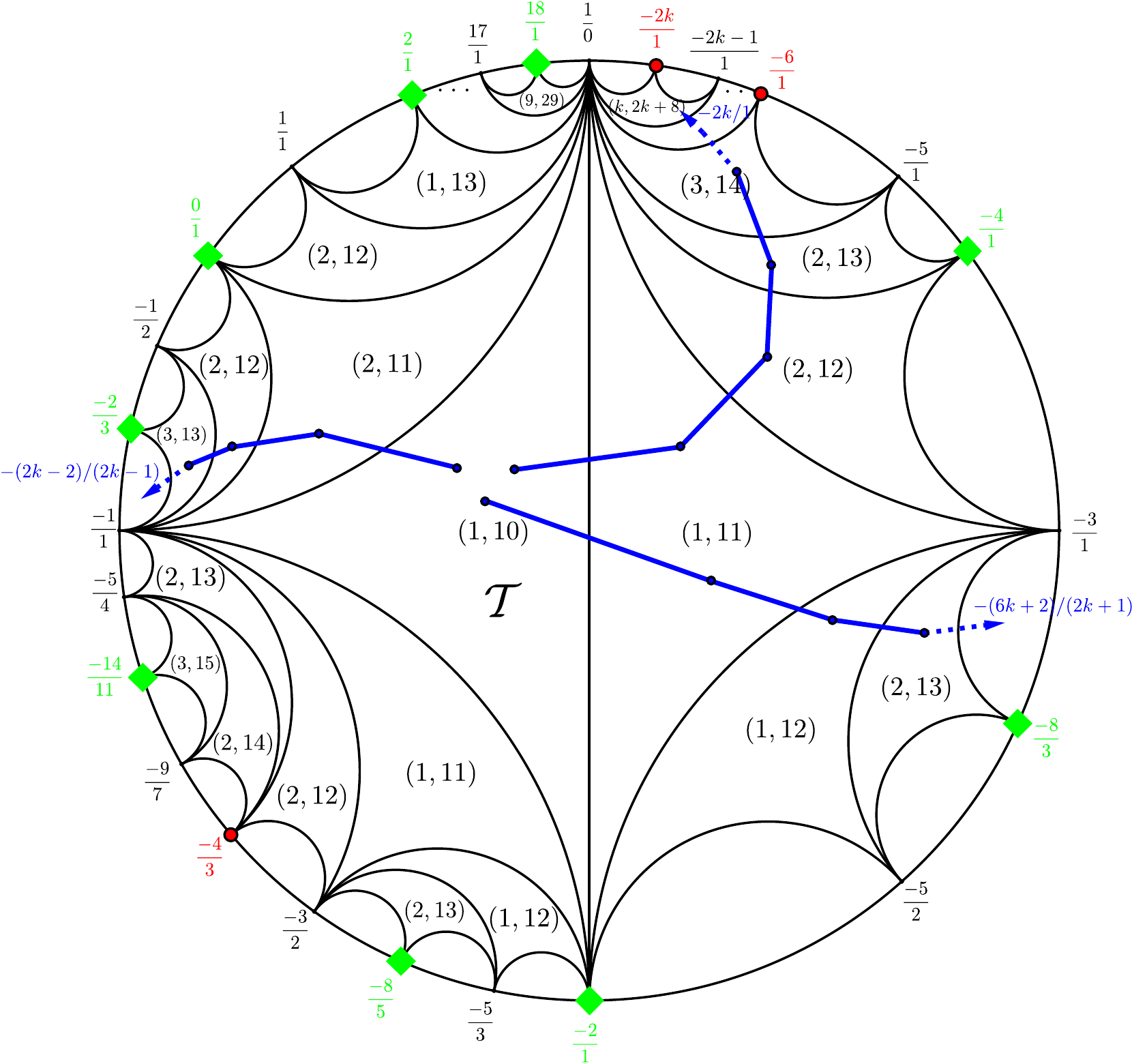}}
\caption{Slope norms and triangulation sizes for the compact core of the Pretzel knot as calculated based on triangulation $\tri$. Green diamonds indicate (even) slopes of fundamental normal surfaces of $\tri.$}
 \label{fig:pretzel}
\end{figure}

The triangulation $\tri$ has $75$ fundamental normal surfaces in standard coordinates. These contain the orientable Seifert surface, and non-orientable surfaces with boundary a single essential curve of eight distinct slopes. Their boundary patterns, boundary slopes, maximum Euler characteristic, and slope norm are summarised in \Cref{tab:pretzelsurfaces}. Once the knot theoretic framing is known, all of this information can be computed directly from the fundamental normal surfaces of $\tri$ and the Farey tessellation following the procedure to compute the slope norm from \cite{JRST21SlopeNorm} and from \Cref{ssec:lowerbound,ssec:upperbound}.

\begin{figure}[htb]
\begin{center}

\begin{tabular}{l|clrc}

or. & $\partial$-pattern & $\chi$ & $\partial$-slope $\alpha$ & $(a,b)$ \\ 

\hline
no & $(1,1,0)$    & $-1$ & $-2/1$   & $(1,10)$ \\
no & $(1,1,2)$    & $-2$ & $ 0/1$   & $(2,11)$ \\
no & $(5,3,2)$    & $-6$ & $-8/5$   & $(2,13)$ \\
no & $(1,3,2)$    & $-2$ & $-4/1$   & $(2,12)$ \\
no & $(1,3,4)$    & $-1$ & $2/1$    & $(1,13)$ \\
no & $(3,1,4)$    & $-3$ & $-2/3$   & $(3,13)$ \\
no & $(3,5,2)$    & $-4$ & $-8/3$   & $(2,13)$ \\
no & $(11,3,8)$   & $-9$ & $-14/11$ & $(3,15)$ \\
\hline
yes & $(1,19,20)$ & $-9$ & $18/1$   & $(9,29)$ 
\end{tabular}
\end{center}
\caption{Boundary pattern, Euler characteristic, boundary slope of fundamental normal surfaces of triangulation $\tri$ of the compact core of the Pretzel knot. Rightmost column: tuple $(a,b)$ of slope norm and upper bound for complexity for triangulations with even boundary edge of the given slope. \label{tab:pretzelsurfaces}}
\end{figure}

From \Cref{fig:pretzel} and extensions into the Farey tessellation we can now directly read off lower and upper bounds for $c(M(\alpha))$, where $\alpha$ is any given even slope $\alpha$:

\begin{enumerate}
  \item Layer on top of $\tri_\partial$ along the unique shortest path in the dual of the Farey tessellation from the base triangle, shaded grey in \Cref{fig:pretzel}, to one layering before the target triangle. That is, one layering before the first triangle containing the target slope $\alpha$ as one of its ideal vertices. 
  The result is a triangulation $\tri'$ with number of tetrahedra ten plus number of layerings.
  \item Fold over the even boundary edge of $\tri'$ to obtain a triangulation of $M(\alpha).$
  \item The $\mathbb{Z}_2$-norm of the unique $\mathbb{Z}_2$--torsion class of $M(\alpha)$ is one less than the slope norm in the target triangle.
  \item The difference of twice the $\mathbb{Z}_2$-norm plus two (if our triangulation is not a balanced lens space) and the size of $\tri'$ (one less than the upper bound recorded in the target triangle) yields the gap up to which we can determine $c(M(\alpha)).$
\end{enumerate}

From the above calculations, we deduce the upper and lower bounds in complexity for infinite families of Dehn fillings of $M.$ In particular, this gives the bounds stated in \Cref{eq:pretzel_bounds1,eq:pretzel_bounds2,eq:pretzel_bounds3}.

 The above procedure does not work for $M(0/1).$ Here, we first need to layer once to obtain a different isotopy class in the boundary and then fold back over the even edge. In this case, a better gap can be obtained by starting with a triangulation with a different isotopy class in the boundary.


%
%

\subsection{Even Dehn fillings of the trefoil knot complement}
\label{ssec:trefoil}

In this section we discuss a non-hyperbolic knot. More specifically, we look at three infinite families of even Dehn fillings of the trefoil knot complement. For each of them we can determine their complexity up to a gap of two.

We start with the $2$-tetrahedra triangulation of the right-handed trefoil knot complement $M$ with {\em Regina} isomorphism signature \texttt{cPcbbbadu}. A search through the Pachner graph yields two triangulations of the compact core of $M$ with four tetrahedra. Their {\em Regina} isomorphism signatures are \texttt{eHLObcdddwun} and \texttt{eHLObcdddwuj} respectively. For the remainder of this section we refer to them as $\tri_1$ and $\tri_2$. Both triangulations are $0$-efficient.

\begin{figure}
\centerline{\includegraphics[width=.8\textwidth]{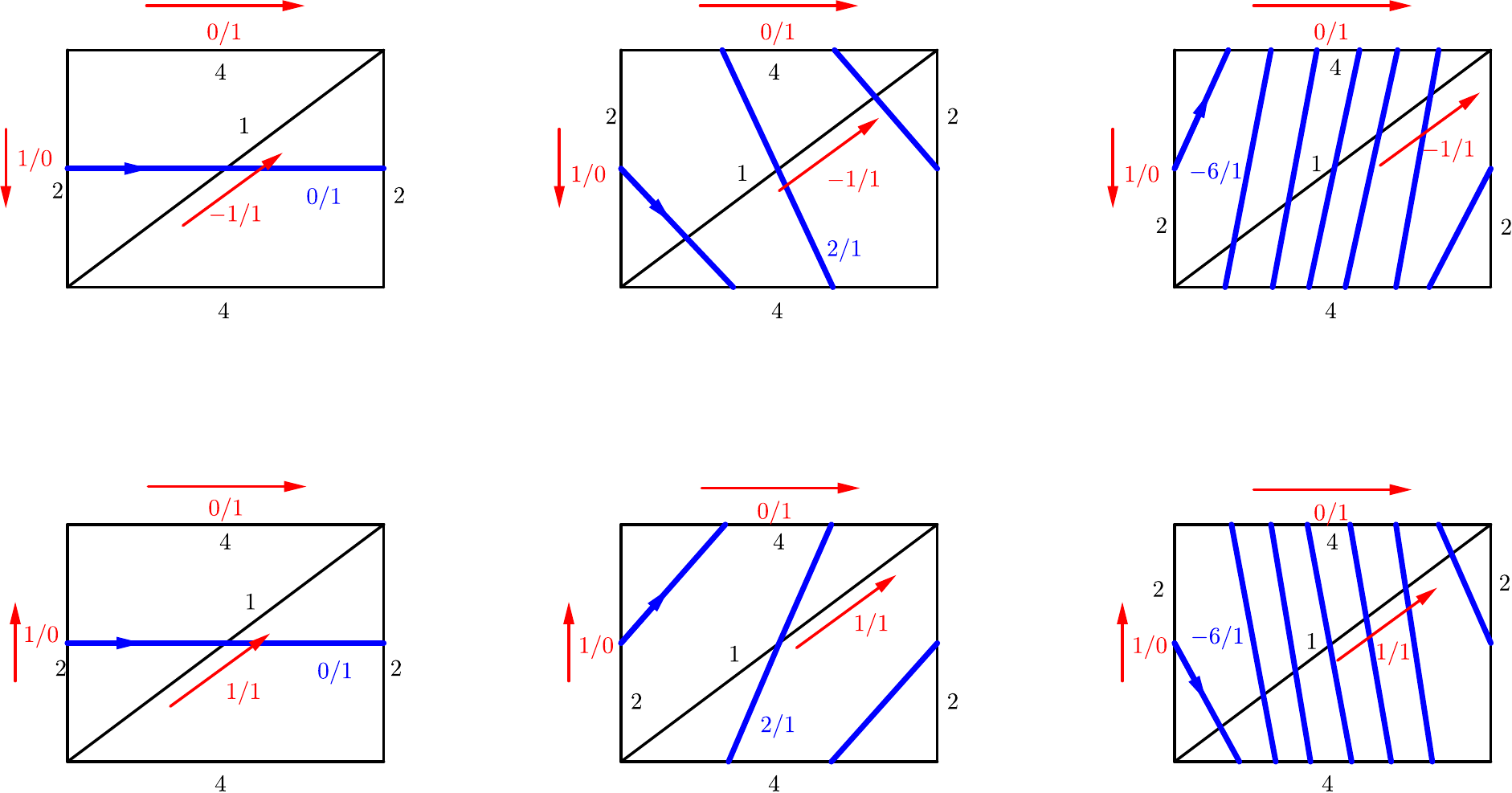}}
\caption{Boundary patterns and choices for framings for $\partial \tri_1$ (top) and $\partial \tri_2$ (bottom), triangulations of the compact core of the trefoil knot. The choices for longitude and meridian are topoligically equivalent for $\partial \tri_1$ and $\partial \tri_2$.}
 \label{fig:trefoil}
\end{figure}

See \Cref{fig:trefoil} for consistent choices of framings, and boundary patterns of fundamental normal surfaces for both $\tri_1$ and $\tri_2$. See \Cref{fig:trefoil_farey} for a marked Farey tessellation containing slope norms (as computed via \cite[Theorem 5]{JRST21SlopeNorm}) and triangulation sizes based on layering on $\tri_1$ and $\tri_2$ respectively. 

\begin{figure}
\centerline{\includegraphics[width=0.7\textwidth]{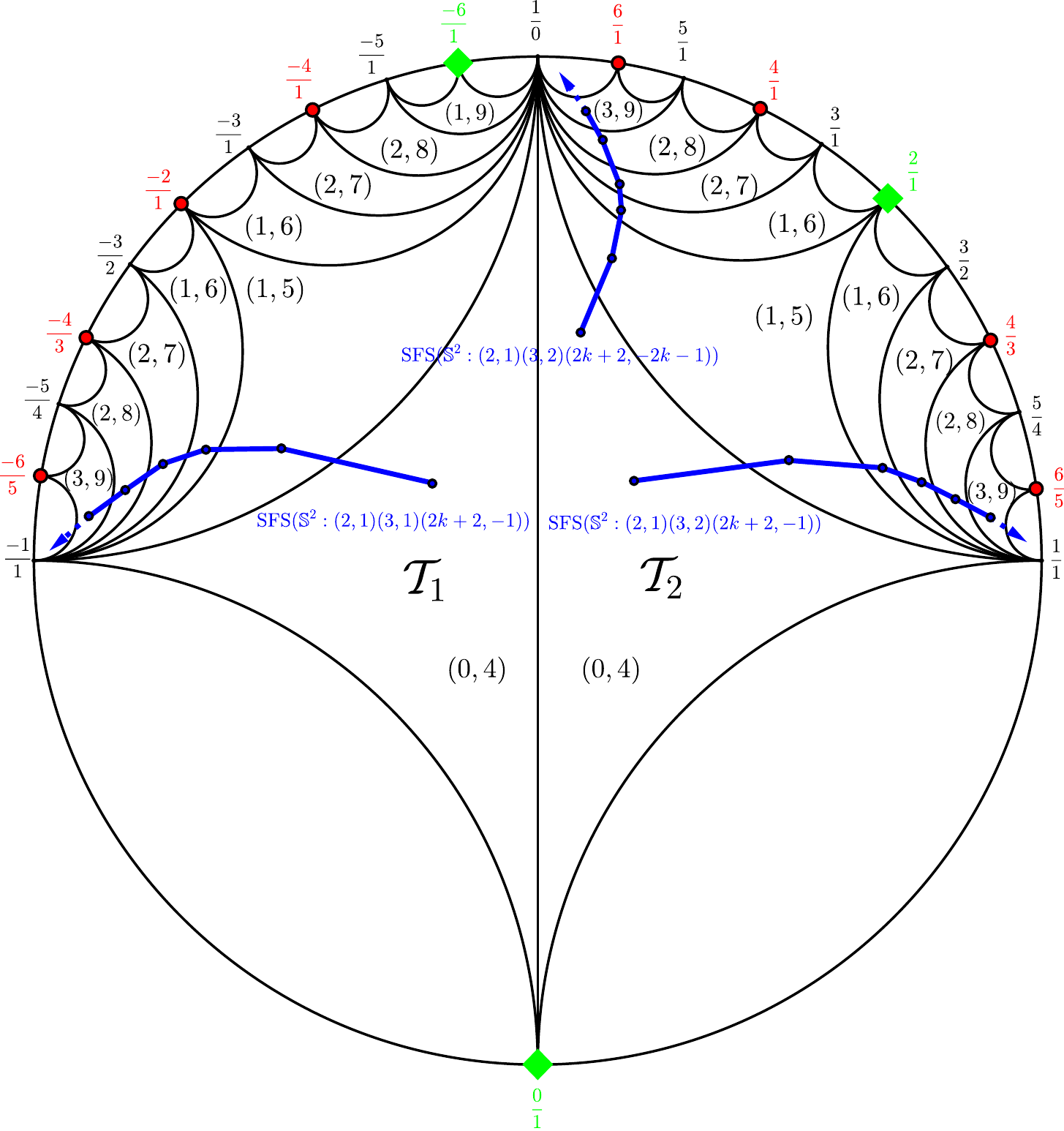}}
\caption{Slope norms and triangulation sizes giving lower and upper bounds for three infinite families of even Dehn fillings of the trefoil complement. Slopes of fundamental normal surfaces in $\tri_1$ and $\tri_2$ are marked in green. The slope of the Seifert surface is $-6/1$.}
 \label{fig:trefoil_farey}
\end{figure}

Starting at an ideal triangle associated to the isotopy class of the boundary of either $\tri_1$ or $\tri_2$, there are a total of three infinite paths through the dual graph of the Farey tesselation with a gap of $b - 2a - 1 = 2$. As in previous examples, we describe these families in terms of their filling slopes $\alpha_k = \alpha \oplus 2k \beta $. The layerings are determined by starting at one of the two base ideal triangles in \Cref{fig:trefoil_farey} and following the path in the dual of the Farey tessellation around the ideal vertex with label $\beta$. In each step, a line of the Farey tessellation is crossed into a new ideal triangle. To obtain a triangulation with boundary of isotopy class the class of the new ideal triangle, we layer over the boundary edge of the existing triangulation with slope the label of the opposite ideal vertex in the old ideal triangle.

\begin{itemize}
  \item The first family is given by $\alpha_k = (-2/1) \oplus 2k (-1/1) $, $k\geq 0$. We have for the topological type $M(\alpha_k) = \operatorname{SFS} (\mathbb{S}^2 : (2,1),(3,1),(2k+2,-1))$. The single $\mathbb{Z}_2$--torsion class of $M(\alpha_k)$ has norm $k$. This leads to $c(M(\alpha_k)) \geq 2k+2$, via the norm, and $c(M(\alpha_k)) \leq 2k+4$ via the layering construction. 
  \item The second family is given by $\alpha_k = (2/1) \oplus 2k (1/0) $, $k\geq 0$. Here, the topological type is $M(\alpha_k) = \operatorname{SFS} (\mathbb{S}^2 : (2,1),(3,2),(2k+2,-2k-1))$, the norm is, again, $k$, and we have for complexity $2k+2 \leq c(M(\alpha_k)) \leq 2k+4$. 
  \item The third family is given by $\alpha_k = (2/1) \oplus 2k (1/1) $, $k\geq 0$. The topological type is $M(\alpha_k) = \operatorname{SFS} (\mathbb{S}^2 : (2,1),(3,2),(2k+2,-1))$, the norm is $k$, and we have for complexity $2k+2 \leq c(M(\alpha_k)) \leq 2k+4$. 
\end{itemize}

In all three cases, the upper bound is conjectured to be the actual complexity. 

The three walks in the dual of the Farey tessellation corresponding to the above families are marked in \Cref{fig:trefoil_farey}. For the first family, we start with triangulation $\tri_1$, while for the other two families we start with $\tri_2$. Note that family $M(\alpha_k)$ with $\alpha_k = (-2/1) \oplus 2k (1/0) $ has a larger gap due to the Seifert surface being on this path. This reduces the $\mathbb{Z}_2$-norm and hence the lower bound in complexity for subsequent members of the associated infinite family of Dehn fillings.


\bibliographystyle{plain}
\bibliography{references}

\begin{thebibliography}{10}

\bibitem{regina}
Benjamin~A. Burton, Ryan Budney, William Pettersson, et~al.
\newblock Regina: Software for low-dimensional topology.
\newblock {\tt http://\allowbreak regina-normal.\allowbreak github.\allowbreak
  io/}, 1999--2021.

\bibitem{Burton-Weber-2012}
Benjamin~A. Burton, J.~Hyam Rubinstein, and Stephan Tillmann.
\newblock The {W}eber-{S}eifert dodecahedral space is non-{H}aken.
\newblock {\em Trans. Amer. Math. Soc.}, 364(2):911--932, 2012.

\bibitem{SnapPy}
Marc Culler, Nathan~M. Dunfield, Matthias Goerner, and Jeffrey~R. Weeks.
\newblock Snap{P}y, a computer program for studying the geometry and topology
  of $3$-manifolds.
\newblock Available at \url{https://snappy.math.uic.edu} (Nov 2021).

\bibitem{Dunfield20Exceptional}
Nathan~M. Dunfield.
\newblock A census of exceptional {D}ehn fillings.
\newblock In {\em Characters in low-dimensional topology}, volume 760 of {\em
  Contemp. Math.}, pages 143--155. Amer. Math. Soc., [Providence], RI, [2020]
  \copyright 2020.

\bibitem{Fominykh-census-2016}
Evgeny Fominykh, Stavros Garoufalidis, Matthias Goerner, Vladimir Tarkaev, and
  Andrei Vesnin.
\newblock A census of tetrahedral hyperbolic manifolds.
\newblock {\em Exp. Math.}, 25(4):466--481, 2016.

\bibitem{Frigerio03Complexity}
Roberto Frigerio, Bruno Martelli, and Carlo Petronio.
\newblock {Dehn Filling of Cusped Hyperbolic 3-Manifolds with Geodesic
  Boundary}.
\newblock {\em J. Differential Geom.}, 64(3):425--455, 2003.

\bibitem{HLP}
Joel Hass, Jeffrey~C. Lagarias, and Nicholas Pippenger.
\newblock The computational complexity of knot and link problems.
\newblock {\em J. ACM}, 46(2):185--211, 1999.

\bibitem{Hatcher-boundary-1982}
Allen~E. Hatcher.
\newblock On the boundary curves of incompressible surfaces.
\newblock {\em Pacific J. Math.}, 99(2):373--377, 1982.

\bibitem{Ishikawa-construction-2016}
Masaharu Ishikawa and Keisuke Nemoto.
\newblock Construction of spines of two-bridge link complements and upper
  bounds of their {M}atveev complexities.
\newblock {\em Hiroshima Math. J.}, 46(2):149--162, 2016.

\bibitem{JJSTBoundsOnNormSurfs}
William Jaco, Jesse Johnson, Jonathan Spreer, and Stephan Tillmann.
\newblock Bounds for the genus of a normal surface.
\newblock {\em Geom. Topol.}, 20(3):1625--1671, 2016.

\bibitem{jaco03-0-inflations}
William Jaco and J.~Hyam Rubinstein.
\newblock Inflations of ideal triangulations.
\newblock {\em Adv. Math.}, 267:176--224, 2014.

\bibitem{Jaco-norm-2020}
William Jaco, J.~Hyam Rubinstein, Jonathan Spreer, and Stephan Tillmann.
\newblock {$\Bbb{Z}_2$}-{T}hurston norm and complexity of 3-manifolds, {II}.
\newblock {\em Algebr. Geom. Topol.}, 20(1):503--529, 2020.

\bibitem{Jaco-ideal-2020}
William Jaco, J.~Hyam Rubinstein, Jonathan Spreer, and Stephan Tillmann.
\newblock On minimal ideal triangulations of cusped hyperbolic 3-manifolds.
\newblock {\em J. Topol.}, 13(1):308--342, 2020.

\bibitem{JRST21SlopeNorm}
William Jaco, J.~Hyam Rubinstein, Jonathan Spreer, and Stephan Tillmann.
\newblock Slope norm and an algorithm to compute the crosscap number.
\newblock arXiv:2108.07599, 36 pages, 14 figures, 2 tables, 2021.

\bibitem{Jaco-minimal-2009}
William Jaco, J.~Hyam Rubinstein, and Stephan Tillmann.
\newblock Minimal triangulations for an infinite family of lens spaces.
\newblock {\em J. Topol.}, 2(1):157--180, 2009.

\bibitem{Jaco-2013-Norm-Pt1}
William Jaco, J.~Hyam Rubinstein, and Stephan Tillmann.
\newblock {$\Bbb{Z}_2$}-{T}hurston norm and complexity of {$3$}-manifolds.
\newblock {\em Math. Ann.}, 356(1):1--22, 2013.

\bibitem{Jaco-decision-2003}
William Jaco and Eric Sedgwick.
\newblock Decision problems in the space of {D}ehn fillings.
\newblock {\em Topology}, 42(4):845--906, 2003.

\bibitem{Lackenby-2019-complexity}
Marc Lackenby and Jessica~S. Purcell.
\newblock The triangulation complexity of fibred 3-manifolds.
\newblock arXiv:1910.10914, 71 pages, 31 figures, 2019.

\bibitem{Lickorish-1962-LWTheorem1}
W.~B.~R. Lickorish.
\newblock A representation of orientable combinatorial {$3$}-manifolds.
\newblock {\em Ann. of Math. (2)}, 76:531--540, 1962.

\bibitem{Matveev-2001-homologygroups}
S.~V. Matveev and E.~L. Pervova.
\newblock Lower bounds for the complexity of three-dimensional manifolds.
\newblock {\em Dokl. Akad. Nauk}, 378(2):151--152, 2001.

\bibitem{Matveev-2020-tabulation}
S.~V. Matveev and V.~V. Tarkaev.
\newblock Recognition and tabulation of 3-manifolds up to complexity 13.
\newblock {\em Chebyshevski\u{\i} Sb.}, 21(2):290--300, 2020.

\bibitem{Matveev-2009-asymptotic}
Sergei Matveev, Carlo Petronio, and Andrei Vesnin.
\newblock Two-sided asymptotic bounds for the complexity of some closed
  hyperbolic three-manifolds.
\newblock {\em J. Aust. Math. Soc.}, 86(2):205--219, 2009.

\bibitem{Matveev-complexity-1990}
Sergei~V. Matveev.
\newblock Complexity theory of three-dimensional manifolds.
\newblock {\em Acta Appl. Math.}, 19(2):101--130, 1990.

\bibitem{Pervova-2008-algebraiccomplexity}
Ekaterina Pervova and Carlo Petronio.
\newblock Complexity and {$T$}-invariant of abelian and {M}ilnor groups, and
  complexity of 3-manifolds.
\newblock {\em Math. Nachr.}, 281(8):1182--1195, 2008.

\bibitem{Petronio-2009-volumecomputations}
Carlo Petronio and Andrei Vesnin.
\newblock Two-sided bounds for the complexity of cyclic branched coverings of
  two-bridge links.
\newblock {\em Osaka J. Math.}, 46(4):1077--1095, 2009.

\bibitem{RST21NewFamily}
J.~Hyam Rubinstein, Jonathan Spreer, and Stephan Tillmann.
\newblock A new family of minimal ideal triangulations of cusped hyperbolic
  3-manifolds.
\newblock arXiv:2112.01654, 18 pages, 10 figures, 3 tables, 2021.

\bibitem{Thurston-notes}
William~P Thurston.
\newblock The geometry and topology of three-manifolds.
\newblock \url{http://msri.org/publications/books/gt3m/}, 1978.

\bibitem{Vesnin-three-2014}
Andrei~Yu. Vesnin, Vladimir~V. Tarkaev, and Evgeny~A. Fominykh.
\newblock Three-dimensional hyperbolic manifolds with cusps of complexity 10
  that have maximal volume.
\newblock {\em Tr. Inst. Mat. Mekh.}, 20(2):74--87, 2014.

\bibitem{Wallace-1963-LWTheorem2}
Andrew~H. Wallace.
\newblock Modifications and cobounding manifolds. {IV}.
\newblock {\em J. Math. Mech.}, 12:445--484, 1963.

\bibitem{handbookknottheory}
Jeff Weeks.
\newblock Computation of hyperbolic structures in knot theory.
\newblock In {\em Handbook of knot theory}. Elsevier, 2005.

\end{thebibliography}




\address{William Jaco\\Department of Mathematics, Oklahoma State University, Stillwater, OK 74078-1058, USA\\{jaco@math.okstate.edu}\\-----}

\address{J. Hyam Rubinstein\\School of Mathematics and Statistics, The University of Melbourne, VIC 3010, Australia\\
{joachim@unimelb.edu.au}\\----- }

\address{Jonathan Spreer\\School of Mathematics and Statistics F07, The University of Sydney, NSW 2006 Australia\\{jonathan.spreer@sydney.edu.au\\-----}}

\address{Stephan Tillmann\\School of Mathematics and Statistics F07, The University of Sydney, NSW 2006 Australia\\{stephan.tillmann@sydney.edu.au}}

\Addresses
                                                      
\end{document}